\documentclass[11pt,a4paper]{article}

\usepackage[english]{babel}
\usepackage[utf8]{inputenc}
\usepackage{graphicx}
\usepackage{amsmath,amsthm,amssymb,bbm}
\usepackage{hyperref}
\usepackage{mathtools}
\usepackage{comment}
\usepackage{enumitem}
\usepackage{caption}
\usepackage{aligned-overset}
\usepackage[a4paper,top=3cm,bottom=3cm,left=3cm,right=3cm,marginparwidth=1.75cm]{geometry}
\usepackage{xcolor}
\usepackage[colorinlistoftodos]{todonotes}
\usepackage{aliascnt}

\newtheorem{theorem}{Theorem}

\newtheorem{prop}[theorem]{Proposition}
\newtheorem{lemma}[theorem]{Lemma}
\newtheorem{corollary}[theorem]{Corollary}

\newcounter{theoremintro}

\newaliascnt{theoremiCT}{theoremintro}
\newtheorem{theoremi}[theoremiCT]{Theorem}
\aliascntresetthe{theoremiCT}

\theoremstyle{definition}
\newtheorem{example}[theorem]{Example}
\newtheorem*{question}{Question}
\newtheorem{definition}[theorem]{Definition}

\newaliascnt{defiiCT}{theoremintro}
\newtheorem{defii}[defiiCT]{Definition}
\aliascntresetthe{defiiCT}

\theoremstyle{remark}

\newtheorem{remark}[theorem]{Remark}
\numberwithin{equation}{section}
\numberwithin{theorem}{section}
\setenumerate[1]{label=(\arabic*)}

\newcommand{\R}{\mathbb{R}}
\newcommand{\N}{\mathbb{N}}

\newcommand{\Z}{\mathbb{Z}}
\newcommand{\C}{\mathbb{C}}
\renewcommand{\tilde}{\widetilde}

\newcommand{\acts}{\curvearrowright}
\newcommand{\e}{\varepsilon}
\newcommand{\ssubset}{\subset\joinrel\subset}
\newcommand{\cstar}{C$^*$}

\title{Dynamical McDuff-type properties for group actions on von Neumann algebras}
\author{\begin{tabular}{ccc} Gábor Szabó\thanks{Supported by BOF project C14/19/088 funded by the research council of KU Leuven, and research project G085020N funded by the Research Foundation Flanders (FWO).} & \& & Lise Wouters\thanks{Supported by PhD-grant 11B6622N funded by the Research Foundation Flanders (FWO).} \end{tabular} \medskip\\
KU Leuven, Department of Mathematics, \\
Celestijnenlaan 200B, 3001 Leuven, Belgium \medskip\\ 
\begin{tabular}{ccc} \texttt{gabor.szabo@kuleuven.be} & & \texttt{lise.wouters@kuleuven.be} \end{tabular}}
\date{}

\begin{document}
\maketitle
\begin{abstract}
\noindent We consider the notion of strong self-absorption for continuous actions of locally compact groups on the hyperfinite II$_1$ factor and characterize when such an action is tensorially absorbed by another given action on any separably acting von Neumann algebra.
This extends the well-known McDuff property for von Neumann algebras and is analogous to the core theorems around strongly self-absorbing \cstar-dynamics.
Given a countable discrete group $G$ and an amenable action $G\curvearrowright M$ on any separably acting semi-finite von Neumann algebra, we establish a type of measurable local-to-global principle: If a given strongly self-absorbing $G$-action is suitably absorbed at the level of each fibre in the direct integral decomposition of $M$, then it is tensorially absorbed by the action on $M$.
As a direct application of Ocneanu's theorem, we deduce that if $M$ has the McDuff property, then every amenable $G$-action on $M$ has the equivariant McDuff property, regardless whether $M$ is assumed to be injective or not.
By employing Tomita--Takesaki theory, we can extend the latter result to the general case where $M$ is not assumed to be semi-finite.
\end{abstract}

\setcounter{tocdepth}{1}
\renewcommand{\baselinestretch}{0.5}\normalsize
\tableofcontents
\renewcommand{\baselinestretch}{1.00}\normalsize


\addcontentsline{toc}{section}{Introduction}
\section*{Introduction}

The classification problem for group actions on operator algebras has a long history dating back to the hallmark results of Connes \cite{Connes73, Connes74, Connes76} about injective factors that employed ideas of a dynamical nature in a crucial way.
From that point, the subsequent $\mathrm{W}^*$-algebraic developments with a dynamical flavour branched off into related but methodologically distinct lines of investigation.
There is on the one hand the classification of amenable group actions on injective factors, which was started by Connes \cite{Connes75, Connes77} for cyclic groups and continued in works of Jones \cite{Jones80}, Ocneanu \cite{Ocneanu85} and others \cite{SutherlandTakesaki89, KawahigashiSutherlandTakesaki92, KatayamaSutherlandTakesaki98, Masuda13}.
When the acting group is nonamenable, then the structure of its actions on the hyperfinite II$_1$ factor is already less managable, as a theorem of Jones \cite{Jones83oc} and other subsequent stronger \emph{no-go theorems} such as \cite{Popa06, BrothierVaes15} later demonstrated.
On the other hand, there is Popa's notion of amenability for subfactors of finite Jones index \cite{Jones83}, for their principal graph and standard invariant, introduced in several equivalent ways in \cite{Popa94, Popa97survey, Popa99, Popa23}, and his classification result showing that any subfactor of the hyperfinite II$_1$ factor that has amenable graph is completely determined by its standard invariant \cite{Popa95survey, Popa97survey, Popa99, Popa23}.
The latter result is also known to recover Ocneanu's theorem for actions of finitely generated amenable groups, as it has been observed in \cite[Subsection 5.1.5]{Popa94} that such actions can be encoded in the framework of finite index subfactors.
We shall leave it at the type II$_1$ case for this direction of classification, as giving a full account of further developments in the infinite case is beyond the scope of this introduction.
In terms of methodology, the major difference between these two branches of research is that while the former uses Connes' classification as a black box in conjunction with a \emph{model action splitting argument}, the methods developed for the latter recover the uniqueness of the injective II$_1$ factor as a consequence rather than needing it as an ingredient.
These two branches were later also unified into one common theory, encompassing the classification of group actions on subfactors, through works such as \cite{Popa92, Popa10}.

Connes' article \cite{Connes76} was the first of many influential works in operator algebras to make use of a kind of touchstone object (in his case the hyperfinite II$_1$ factor $\mathcal{R}$) and to begin the classification approach by showing that every object to be classified absorbs this touchstone object.
More specifically, Connes' approach begins with a structural result asserting that every injective II$_1$ factor is \emph{McDuff} --- i.e., tensorially absorbs $\mathcal R$; see \cite{McDuff70} --- which is then used further to show that each such factor is isomorphic to $\mathcal R$.
In Ocneanu's work to classify outer actions of an arbitrary countable amenable group $G$ on $\mathcal{R}$, he likewise proves at an early point in the theory that each such action (even without assuming outerness) absorbs the trivial $G$-action on $\mathcal R$ tensorially (in the terminology of that time, Ocneanu `split off the trivial action on $\mathcal R$'), which he then exploits for his actual classification theorem.
Although one would generally need injectivity of a factor to arrive at a satisfactory classification theory about (outer) $G$-actions on it, this early part of Ocneanu's theory works in general.
The precise statement and a (comparably) self-contained proof using the methods of this article, which we included for the reader's convenience, is contained in Theorem~\ref{theorem:model-absorption}.

If one is concerned with \cstar-algebraic considerations related to the \emph{equivariant Jiang--Su stability problem} (see \cite[Conjecture A]{Szabo21si}), the current methods always find a way to exploit Oceanu's aforementioned theorem in one form or another, usually involving to some degree Matui--Sato's property (SI) technique \cite{MatuiSato12acta, MatuiSato12, MatuiSato14, Sato19}.
Looking at the state-of-the-art at present \cite{GardellaHirshbergVaccaro, Wouters21}, the key difficulties arise from pushing these methods to the case where a group action $G\curvearrowright A$ on a \cstar-algebra induces a complicated $G$-action on the traces of $A$.
In particular, it is generally insufficient for such considerations to only understand $G$-actions on $\mathcal{R}$, but one rather needs to have control over $G$-actions on more general tracial von Neumann algebras.
This \cstar-algebraically motivated line of investigation led us to ask the following question that is intrinsic to von Neumann algebras:

\begin{question}
Let $G$ be a countable amenable group and $M$ a separably acting finite von Neumann algebra with $M\cong M\bar{\otimes}\mathcal R$.
Is it true that every action $\alpha: G\curvearrowright M$ is cocycle conjugate to $\alpha\otimes\operatorname{id}_{\mathcal R}: G\curvearrowright M\bar{\otimes}\mathcal R$?
\end{question}

Although Ocneanu's original results confirm this when $M$ is a factor,\footnote{While Ocneanu's work \cite{Ocneanu85} only contains an explicit proof for so-called centrally free actions $\alpha$, his comment following \cite[Theorem 1.2]{Ocneanu85} suggests an alternative approach to avoid this assumption. In several papers, the more general version without central freeness is also attributed to Ocneanu.} it turned out to be not so straightforward to resolve this question, despite common folklore wisdom in the subject suggesting that the factor case ought to imply the general case.
Some classification results in the literature \cite{JonesTakesaki84, SutherlandTakesaki85} imply that the above has a positive answer when $M$ is injective, but relying on this has two drawbacks.
Firstly, the question we are trying to answer is by design much weaker than a hard classification result, so it would be desirable to have a proof not relying on such a powerful theorem, in particular when an assumption such as injectivity may not even be needed.
Secondly, there is a subtle gap in the proof of \cite[Lemma 4.2]{SutherlandTakesaki85}.
We are indebted to Stefaan Vaes for pointing this out to us in the context of the above question and for outlining a sketch of proof on how to correct this, which became a sort of blueprint for the main result of the fourth section.

In contemporary research by \cstar-algebraists, the aforementioned results by Sutherland--Takesaki are still used to provide a partial answer to the above question, for example in \cite{GardellaHirshbergVaccaro}. In light of the previous discussion, the present article aims to give a self-contained and --- dare we say also relatively elementary --- approach to answer this question instead.
In fact we can treat it in greater generality than posed above, without restrictions on the type of $M$ and in the setting of amenable actions of arbitrary discrete groups.
The following can be viewed as our main result; see Theorem~\ref{thm:general-amenable-McDuff}.

\begin{theoremi} \label{theorem-A}
Let $G$ be a countable discrete group and $M$ a von Neumann algebra with separable predual such that $M \cong M\bar{\otimes} \mathcal{R}$. 
Then every amenable action $\alpha\colon G \acts M$ is cocycle conjugate to $\alpha \otimes \mathrm{id}_\mathcal{R}\colon G\curvearrowright M\bar{\otimes} \mathcal{R}$.
\end{theoremi}

As we will comment later (Remark~\ref{rem:compact-fails-A}), this phenomenon cannot be expected outside the discrete case and in fact the statement fails if we insert the circle group in place of $G$.
Along the way, our methodology employs dynamical variants of McDuff-type properties analogous to the theory of strongly self-absorbing \cstar-dynamics \cite{Szabo18ssa}, which can and is treated in the more general setting of continuous actions of locally compact groups; see Definitions~\ref{def:strong-absorption} and \ref{def:ssa-action}.

\begin{defii}
Let $G$ be a second-countable locally compact group.
An action $\delta: G\curvearrowright\mathcal R$ is called \emph{strongly self-absorbing}, if there exists an isomorphism $\Phi: \mathcal R\to\mathcal R\bar{\otimes}\mathcal R$, a $(\delta\otimes\delta)$-cocycle $\mathbb{U}: G\to\mathcal{U}(\mathcal R\bar{\otimes}\mathcal R)$ and a sequence of unitaries $v_n\in\mathcal{U}(\mathcal R\bar{\otimes}\mathcal R)$ such that
\[
v_n(x\otimes 1_{\mathcal R})v_n^* \to \Phi(x) \quad\text{and}\quad v_n(\delta\otimes\delta)_g(v_n)^* \to \mathbb{U}_g
\]
in the strong operator topology for all $x\in\mathcal R$ and $g\in G$, the latter uniformly over compact subsets in $G$.
\end{defii}

For such actions we prove the following dynamical generalization of McDuff's famous theorem \cite{McDuff70}; see Theorem~\ref{thm:general-McDuff}.

\begin{theoremi} \label{theorem-C}
Let $G$ be a second-countable locally compact group.
Let $\alpha: G \acts M$ be an action on a von Neumann algebra with separable predual and let $\delta: G \acts \mathcal{R}$ be a strongly self-absorbing action on the hyperfinite II$_1$ factor.
Then $\alpha$ is cocycle conjugate to $\alpha\otimes\delta$ if and only if there exists a unital equivariant $*$-homomorphism $(\mathcal R,\delta) \to (M_{\omega,\alpha},\alpha_\omega)$, where the latter denotes the induced $G$-action on the asymptotic centralizer algebra of $M$.
\end{theoremi}

Our initial methodology inspired by the theory of \cstar-dynamics is only well-suited to build all the aforementioned (and other related) theory in the setting of (actions on) semi-finite von Neumann algebras.
After the first preliminary section, the second and third section are dedicated to proving Theorem~\ref{theorem-C} in the special case of semi-finite von Neumann algebras.
The fourth section then builds on some of this theory, combined with the original ideas by Sutherland--Takesaki \cite{SutherlandTakesaki85} related to disintegrating a $G$-action to an action of its transformation groupoid induced on the center. This culminates in our main technical result of that section --- a kind of measurable local-to-global principle for absorbing a given strongly self-absorbing action, Theorem~\ref{thm:main_technical} --- which is then used to prove a stronger version of Theorem~\ref{theorem-A} for actions on semi-finite von Neumann algebras. 
 
The general main results are then obtained in the fifth section with the help of Tomita--Takesaki theory.
It is in this step that it becomes obvious why we want to treat Theorem~\ref{theorem-C} beyond the case of discrete groups.
Namely, if $\alpha\colon G\curvearrowright M$ is an action as in Theorem~\ref{theorem-A} on a von Neumann algebra that is not semi-finite, we may consider the extended action $\tilde{\alpha}\colon G\curvearrowright\tilde{M}$ on the (semi-finite) continuous core.
However, in order to conclude that $\alpha$ absorbs $\delta$ with the help of Tomita--Takesaki theory, it is not sufficient to argue that $\tilde{\alpha}$ absorbs $\delta$, but one actually needs to verify this absorption for certain enlargements of these actions to continuous actions of $G\times\mathbb{R}$, which in any case requires Theorem~\ref{theorem-C} for non-discrete groups.
Fortunately this can all be arranged to work and we end the last section with the proofs of our main results.


\section{Preliminaries}

Throughout the paper, $\omega$ denotes a fixed free ultrafilter on $\N$ and $G$ denotes a second-countable, locally compact group.
Let $M$ be a von Neumann algebra with predual $M_*$. For $x \in M$ and $\phi \in M_*$ we define elements $x\phi, \phi x$ and $[x,\phi] \in M_*$ by $(x\phi)(y) = \phi(yx)$, $(\phi x)(y) = \phi(xy)$ for all $y \in M$ and $[x,\phi] = x\phi - \phi x$. Moreover, for $x \in M$ and $\phi \in M_*$ we set $\|x\|_{\phi} = \phi(x^*x)^{1/2}$ and $\|x\|_{\phi}^\sharp = \phi(x^*x + xx^*)^{1/2}$. When $\phi$ is a faithful normal state, the norms $\|\cdot\|_\phi$ and $\|\cdot\|_\phi^\sharp$ induce the strong and strong-$*$ topology on bounded sets, respectively.
More generally, when $\phi$ is a normal weight on $M$, we define $\|x\|_{\phi}:= \phi(x^*x)^{1/2}$ for all $x$ contained in the left-ideal
\[\{x  \in M \mid \phi(x^*x) < \infty\}.\]
Recall that $M$ is called $\sigma$-finite if it admits a faithful normal state.
Throughout, the symbol $\mathcal R$ is used to denote the hyperfinite II$_1$ factor.
A von Neumann algebra $M$ is said to have the McDuff property, if $M$ is isomorphic to $M\bar{\otimes}\mathcal R$.

\subsection{Ultrapowers of von Neumann algebras}

We start with a reminder on the Ocneanu ultrapower of a von Neumann algebra and related concepts.
This originates in \cite[Section 2]{Connes74} and \cite[Section 5]{Ocneanu85}, but the reader is also referred to \cite{AndoHaagerup14} for a thorough exposition on ultrapower constructions.

\begin{definition}
Let $M$ be a $\sigma$-finite von Neumann algebra. We define the subset $\mathcal{I}_\omega(M) \subset \ell^\infty(M)$ by
\begin{align*}
\mathcal{I}_\omega(M) &= \{(x_n)_{n \in \N} \in \ell^\infty(M) \mid x_n \rightarrow 0 \text{ $*$-strongly as } n \rightarrow \omega\}\\
&= \{(x_n)_{n \in \N} \in \ell^\infty(M) \mid \lim_{n \rightarrow \omega}\|x_n\|_{\phi}^\sharp =0 \text{ for some faithful normal state } \phi \text{ on } M\}.
\end{align*}
Denote \[\mathcal{N}_\omega(M)=\{(x_n)_{n \in \N} \in \ell^\infty(M) \mid (x_n)_{n \in \N} \mathcal{I}_\omega(M) \subset \mathcal{I}_\omega(M), \text{ and } \mathcal{I}_\omega(M)(x_n)_{n \in \N} \subset \mathcal{I}_\omega(M)\},\] 
\[\mathcal{C}_\omega(M) =\{(x_n)_{n \in \N} \in \ell^\infty(M) \mid \lim_{n \rightarrow \omega}\|[x_n,\phi]\| = 0 \text{ for all } \phi \in M_*\}.\]
Then \[{\mathcal{I}_\omega(M) \subset \mathcal{C}_\omega(M) \subset \mathcal{N}_\omega(M)}.\] 
The \emph{Ocneanu} ultrapower $M^\omega$ of $M$ is defined as
\[M^\omega := \mathcal{N}_\omega(M)/\mathcal{I}_\omega(M),\]
and the \emph{asymptotic centralizer} $M_\omega$ of $M$ is defined as
\[M_\omega := \mathcal{C}_\omega(M)/\mathcal{I}_\omega(M).\]
These are both von Neumann algebras. 
Any faithful normal state $\phi$ on $M$ induces a faithful normal state $\phi^\omega$ on $M^\omega$ via the formula
\[\phi^\omega((x_n)_{n \in \N}) = \lim_{n \rightarrow \omega} \phi(x_n).\]
The restriction of $\phi^\omega$ to $M_\omega$ is a tracial state.
\end{definition}

\begin{remark}
Since the constant sequences are easily seen to be contained in $\mathcal{N}_\omega(M)$, one considers $M$ as a subalgebra of $M^\omega$.
If $\lim_{n \rightarrow \omega}\|[x_n,\phi]\| = 0$ for all $\phi \in M_*$, then $\lim_{n \rightarrow \omega}\|[x_n,y]\|_\phi^\sharp = 0$ for all $y \in M$ by  \cite[Proposition 2.8]{Connes74}.
  In this way we get a natural inclusion $M_\omega \subset M^\omega \cap M'$.
   That same proposition also shows that in order to check whether a sequence $(x_n)_n$ in $\ell^\infty(M)$ satisfies $\lim_{n \rightarrow \omega}\|[x_n,\psi]\| = 0$ for all $\psi \in M_*$, it suffices to check if this is true for just a single faithful normal state $\phi$ and to check if $\lim_{n \rightarrow \omega}\|[x_n,y]\|^\sharp_\phi =0$ for all $y \in M$.
    This shows that $M_\omega = M^\omega \cap M'$ whenever $M$ admits a faithful normal tracial state. The same is then true for all semi-finite von Neumann algebras with separable predual (for example by \cite[Lemma~2.8]{MasudaTomatsu16}).
\end{remark}

\begin{definition}
A continuous action $\alpha\colon G \acts M$ of a second-countable locally compact group on a von Neumann algebra is a homomorphism $G \to \mathrm{Aut}(M)$, $g \mapsto \alpha_g$ such that
\[
\lim_{g \rightarrow 1_G}\|\varphi \circ \alpha_g - \varphi\|=0 \text{ for all } \varphi \in M_*.
\]
By \cite[Proposition~X.1.2]{Takesaki03}, this is equivalent to the map being continuous for the point-weak-$*$ (or equivalently, point-weak, point-strong,$\hdots$) topology.
In most contexts we omit the word ``continuous'' as it will be implicitly understood that we consider some actions to be continuous.
In contrast, we will explicitly talk of an algebraic $G$-action when we are considering an action of $G$ viewed as a discrete group.
\end{definition}

Given an action $\alpha\colon G \acts M$, the induced algebraic actions $\alpha^\omega\colon G \rightarrow \mathrm{Aut}(M^\omega)$ and $\alpha_\omega \colon G \rightarrow \mathrm{Aut}(M_\omega)$ are usually not continuous.
The remainder of this subsection is devoted (for lack of a good literature reference) to flesh out the construction of their `largest' von Neumann subalgebras where the action is sufficiently well-behaved for our needs, called the \emph{$(\alpha, \omega)$-equicontinuous parts} (see Definition \ref{def:equicontinuous_parts}).
These constructions and proofs are based on \cite[Section 3]{MasudaTomatsu16}, where the special case $G=\R$ is considered. The general definition appeared first in \cite[2.1]{Tomatsu17}. 

\begin{definition}
Let $M$ be a $\sigma$-finite von Neumann algebra with an action $\alpha\colon G \acts M$.
Fix a faithful normal state $\phi$ on $M$. An element $(x_n)_{n \in \N} \in \ell^\infty(M)$ is called \emph{$(\alpha, \omega)$-equicontinuous} if for every $\e>0$, there exists a set $W \in \omega$ and an open neighborhood $1_G\in U \subset G$ such that
\[
\sup_{n\in W} \sup_{g\in U} \|\alpha_g(x_n) - x_n\|_\phi^\sharp < \e .
\]
We denote the set of $(\alpha, \omega)$-equicontinuous sequences by $\mathcal{E}^\omega_\alpha$.
\end{definition}

\begin{remark}
The definition above does not depend on the faithful normal state chosen.
Whenever $\phi$ and $\psi$ are two faithful normal states on $M$, one has for every $\e>0$ some $\delta>0$ such that for all $x \in (M)_1$, $\|x\|_\phi^\sharp < \delta$ implies $\|x\|_\psi^\sharp < \e$. 
\end{remark}

\begin{lemma}\label{lemma:epsilon_delta_sequences}
Let $M$ be a von Neumann algebra with faithful normal state $\phi$. For all $(x_n)_{n \in \N} \in \mathcal{N}_\omega(M)$ the following holds: 
For any $\e >0$ and compact set $\Psi \subset M_*^+$ there exists a $\delta >0$ and $W \in \omega$ such that if $y \in (M)_1$ and $\|y\|_\phi^\sharp< \delta$, then $\sup_{\psi\in \Psi} \|x_n y\|_{\psi}^\sharp  <\e$ and $\sup_{\psi\in \Psi} \|y x_n\|_{\psi}^\sharp  <\e$ for all $n \in W$. 
\end{lemma}
\begin{proof} We prove this by contradiction. 
Suppose that there exists $\e >0$ and a compact set $\Psi \subset M_*^+$ such that for any $k \in \N$ there exists a $y_k \in (M)_1$ with $\|y_k\|_\phi^\sharp < 1/k$ but the following set belongs to $\omega$: 
\[
A_k := \Big\{ n \in\N \ \Big| \ \sup_{\psi \in \Psi}\|x_ny_k\|_\psi^\sharp\geq \e \text{ or } \sup_{\psi \in \Psi}\|y_k x_n\|_\psi^\sharp \geq \e \Big\}.
\]
Define $W_0 := \N$ and $W_k := A_1 \cap \hdots \cap A_k \cap [k, \infty)$ for $k \geq 1$. 
These all belong to $\omega$. 
For each $n \in \N$ define $k(n)$ as the unique number $k \geq 0$ with $n \in W_k \setminus W_{k+1}$. Put $z_n := y_{k(n)}$ if $k(n) \geq 1$, else put $z_n:=1_M$. 
Note that for all $n \in W_m$ we get that $\|z_n\|_\phi^\sharp = \|y_{k(n)}\|_\phi^\sharp <\frac{1}{k(n)} \leq \frac{1}{m}$. 
Therefore, it holds that $(z_n)_{n\in \N} \in \mathcal{I}_\omega(M)$. 
Since $(x_n)_{n \in \N} \in \mathcal{N}_\omega(M)$, it follows that also $(x_nz_n)_{n\in \N}$ and $(z_nx_n)_{n\in \N}$ belong to $\mathcal{I}_\omega(M)$. 
Hence we get that for all $\psi \in \Psi$ 
\[\lim_{n \rightarrow \omega} \left(\|x_nz_n\|_\psi^\sharp + \|z_nx_n\|_\psi^\sharp\right) = 0.\]
Since $\Psi$ is compact, we also have 
\[\lim_{n \rightarrow \omega} \sup_{\psi \in \Psi}\left(\|x_nz_n\|_\psi^\sharp + \|z_nx_n\|_\psi^\sharp\right) = 0.\]
This gives a contradiction, since our choice of $z_n$ implies that for all $n \in W_1$ 
\[\sup_{\psi \in \Psi}\left(\|x_nz_n\|_\psi^\sharp + \|z_nx_n\|_\psi^\sharp\right)\geq \e.\]
\end{proof}

\begin{lemma}
Let $M$ be a $\sigma$-finite von Neumann algebra with action $\alpha:G \acts M$.
For any two sequences $(x_n)_{n \in \N}, (y_n)_{n \in \N} \in  \mathcal{E}_\alpha^\omega \cap \mathcal{N}_\omega(M)$ it follows that $(x_ny_n)_{n\in \N} \in \mathcal{E}_\alpha^\omega$. 
\end{lemma}
\begin{proof}
Without loss of generality we may assume $\sup_{n \in \N} \|x_n\| \leq \frac{1}{2}$ and $\sup_{n \in \N} \|y_n\| \leq \frac{1}{2}$.
Fix a faithful normal state $\phi$ on $M$. Let $K \subset G$ be a compact neighbourhood of the neutral element.
Take $\e >0$ arbitrarily.
By Lemma \ref{lemma:epsilon_delta_sequences} there exists $\delta>0$ and $W_1 \in \omega$ such that for every $z \in (M)_1$ with $\|z\|_\phi^\sharp < \delta$ one has 
\[
\sup_{g \in K}\|x_n z\|_{\phi \circ \alpha_g}^\sharp< \frac{\e}{2} \text{ and } \|zy_n\|_\phi^\sharp < \frac{\e}{2} \text{ for all } n \in W_1.
\] 
Since $(x_n)_{n\in \N}$ and $(y_n)_{n\in \N}$ both belong to $\mathcal{E}_\alpha^\omega$, we can find an open $U \subset K$ containing the neutral element, and a $W_2 \in \omega$ such that
\[
\sup_{n \in W_2}\sup_{ g \in U}\|\alpha_g(x_n) - x_n\|_\phi^\sharp < \delta \text{, and}
\]
\[\sup_{n \in W_2}\sup_{ g \in U}
\|\alpha_{g^{-1}}(y_n) - y_n\|_\phi^\sharp < \delta.
\]
Then for $g \in U$ and $n \in W_1 \cap W_2$ we have
\begin{align*}
\|\alpha_g(x_n) \alpha_g(y_n) - x_ny_n\|_\phi^\sharp &\leq \|\alpha_g(x_n)(\alpha_g(y_n) - y_n)\|_\phi^\sharp + \|(\alpha_g(x_n) - x_n)y_n\|_\phi^\sharp\\
&= \|x_n(\alpha_{g^{-1}}(y_n) - y_n)\|_{\phi \circ \alpha_g}^\sharp + \|(\alpha_g(x_n) - x_n)y_n\|_\phi^\sharp\\
& < \frac{\e}{2} + \frac{\e}{2} = \e.
\end{align*}
This ends the proof.
\end{proof}

\begin{lemma}
Let $M$ be a $\sigma$-finite von Neumann algebra with an action $\alpha: G \acts M$.
Then:
\begin{enumerate}[label=\textup{(\arabic*)},leftmargin=*]
\item \label{lem:equicont-algebra:1}
Suppose $(x_n)_{n\in \N}, (y_n)_{n\in \N} \in \ell^\infty(M)$ satisfy $(x_n-y_n)_{n\in \N} \in \mathcal{I}_\omega(M)$.
Then $(x_n)_{n\in \N} \in \mathcal{E}^\omega_\alpha$ if and only if $(y_n)_{n\in \N} \in \mathcal{E}^\omega_\alpha$.
\item \label{lem:equicont-algebra:2}
$\mathcal{E}_\alpha^\omega \cap \mathcal{N}_\omega(M)$ is an $\alpha$-invariant C$^*$-subalgebra of $\ell^\infty(M)$. 
\end{enumerate}
\end{lemma}
\begin{proof} Fix a faithful normal state $\phi$ on $M$.
We first prove \ref{lem:equicont-algebra:1}.
Let $\e>0$.
We can choose $W \in \omega$ and an open neighborhood $1_G \in U \subset G$ such that 
\[
\sup_{n\in W} \sup_{g\in U} \|\alpha_g(x_n) - x_n\|_\phi^\sharp < \frac{\e}{2}.
\] 
Without loss of generality we may assume that $K=\overline{U}$ is compact.
Consider $s_n := \sup_{g \in K} \|x_n - y_n\|_{\phi \circ \alpha_g}^\sharp$. Since $K$ is compact and ${(x_n-y_n)_{n\in \N} \in \mathcal{I}_\omega(M)}$, we have $\lim_{n \rightarrow \omega} s_n = 0$.
Hence, after possibly replacing $W$ by a smaller set in the ultrafilter, we can assume that $s_n < \e/4$ for all $n \in W$.
We may conclude for all $g \in U$ and $n \in W$ that
\begin{align*}
\|\alpha_g(y_n) - y_n\|_\phi^\sharp &\leq \|\alpha_g(y_n) - \alpha_g(x_n)\|_\phi^\sharp + \|\alpha_g(x_n)- x_n\|_\phi^\sharp + \|x_n - y_n\|_\phi^\sharp\\
&\leq 2s_n + \frac{\e}{2} < \e.
\end{align*}
Since $\e>0$ was arbitrary, $(y_n)_{n\in \N}$ belongs to $\mathcal{E}_\alpha^\omega$.

Let us prove \ref{lem:equicont-algebra:2}.
Clearly $\mathcal{E}_\alpha^\omega$ is a $*$-closed, norm-closed linear subspace of $\ell^\infty(M)$. 
The previous lemma shows that $\mathcal{E}_\alpha^\omega \cap \mathcal{N}_\omega(M)$ is closed under multiplication. 
To see that $\mathcal{E}_\alpha^\omega$ is $\alpha$-invariant, take $(x_n)_{n \in \N} \in \mathcal{E}_\alpha^\omega$ and $h \in G$. Take $\e >0$. 
We can find an open neighborhood $1_g \in U \subset G$ and $W \in \omega$ such that one has
\[\sup_{n \in W}\sup_{g \in U}\|\alpha_g(x_n) - x_n\|_{\phi \circ \alpha_h}^\sharp < \e.\] 
Then for all $g \in hUh^{-1}$ and $n \in W$ we observe
\[
\|\alpha_g(\alpha_h(x_n)) - \alpha_h(x_n)\|_\phi^\sharp = \|\alpha_{h^{-1}gh}(x_n) - x_n\|_{\phi \circ \alpha_h}^\sharp < \e.
\]
This shows that $(\alpha_h(x_n))_{n\in \N} \in \mathcal{E}_\alpha^\omega$.
\end{proof}

\begin{definition}\label{def:equicontinuous_parts}
Let $M$ be a $\sigma$-finite von Neumann algebra with an action $\alpha\colon G \acts M$. 
We define ${M_\alpha^\omega := (\mathcal{E}_\alpha^\omega \cap \mathcal{N}_\omega(M))/\mathcal{I}_\omega}$ and ${M_{\omega, \alpha}:= M_\alpha^\omega \cap M_\omega}$. 
We call them the \emph{$(\alpha,\omega)$-equicontinuous parts} of $M^\omega$ and $M_\omega$, respectively. 
\end{definition}

\begin{lemma}
Let $M$ be a $\sigma$-finite von Neumann algebra with an action $\alpha\colon G \acts M$.
Then $M_\alpha^\omega$ and $M_{\omega, \alpha}$ are von Neumann algebras.  
\end{lemma}
\begin{proof}
We show that $M_\alpha^\omega$ is a von Neumann algebra by showing that its unit ball is closed with respect to the strong operator topology in $M^\omega$. 
Then it automatically follows that $M_{\omega, \alpha} = M_\alpha^\omega \cap M_\omega$ is also a von Neumann algebra. 
Take a sequence $(X_k)_k \in (M_\alpha^\omega)_1$ that strongly converges to $X \in (M^\omega)_1$. Fix a faithful normal state $\phi$ on $M$ and a compact neighbourhood of the neutral element $K \subset G$. 
Then the function $K \rightarrow (M^\omega)_*$ given by $g \mapsto \phi^\omega \circ \alpha_g^\omega$ is continuous (because $\phi^\omega \circ \alpha_g^\omega = (\phi \circ \alpha_g)^\omega)$. 
Hence, the set $\{\phi^\omega \circ \alpha_g^\omega\}_{g \in K}$ is compact and thus $\lim_{n \rightarrow \infty} \sup_{g \in K} \|X_n - X\|^\sharp_{\phi^\omega \circ \alpha_g^\omega} = 0$.
Fix $\e >0$. 
Pick representing sequences $(x_k(n))_{n\in \N}$ and $(x(n))_{n\in \N}$ for the elements $X_k$ and $X$, respectively, such that $\|x_k(n)\| \leq 1$, $\|x(n)\| \leq 1$, for all $k,n \in \N$.
Then we can find $k_0 \in \N$  and $W_1 \in \omega$ such that 
\[
\sup_{n\in W_1} \sup_{g\in K} \|x_{k_0}(n) - x(n)\|_{\phi\circ \alpha_g}^\sharp < \frac{\e}{3}.
\]
Since $(x_{k_0}(n))_{n\in \N} \in \mathcal{E}_\alpha^\omega$, we can find an open neighborhood $1_G \in U \subset K$ and $W_2 \in \omega$ such that
\[
\sup_{n\in W_2} \sup_{g\in U} \|\alpha_g(x_{k_0}(n)) - x_{k_0}(n)\|_\phi^\sharp < \frac{\e}{3}.
\]
Then for all $g \in U$ and $n \in W_1 \cap W_2$ it holds that
\begin{align*}
\|\alpha_g(x(n)) - x(n)\|_\phi^\sharp & \leq \|x(n) - x_{k_0}(n)\|_{\phi \circ \alpha_g}^\sharp + \|\alpha_g(x_{k_0}(n)) - x_{k_0}(n)\|_\phi^\sharp + \|x_{k_0}(n) - x(n)\|_\phi^\sharp\\
&< \frac{\e}{3}+\frac{\e}{3}+\frac{\e}{3} = \e. 
\end{align*} 
This shows that $(x(n))_{n \in \N} \in \mathcal{E}_\alpha^\omega$, or in other words $X \in M_\alpha^\omega$.
\end{proof}

\begin{lemma}
Let $M$ be a $\sigma$-finite von Neumann algebra with an action $\alpha\colon G \acts M$ of a second-countable locally compact group. 
Then $\alpha^\omega$ restricted to $M_\alpha^\omega$ and $\alpha_\omega$ restricted to $M_{\omega,\alpha}$ are continuous $G$-actions.
\end{lemma}
\begin{proof}
Fix a faithful normal state $\phi$ on $M$. Since $\phi^\omega$ is faithful, $\{a\phi^\omega \mid a \in M_\alpha^\omega\}$ is dense in $(M_\alpha^\omega)_*$. For $a \in M_\alpha^\omega$ and $g \in G$ one has
\begin{align*}
\|(a\phi^\omega)  \circ \alpha^\omega_g - a \phi^\omega \|_{(M_\alpha^\omega)_*} & \leq \|\alpha_{g^{-1}}^\omega(a)(\phi^\omega \circ \alpha_g^\omega - \phi^\omega)\|_{(M_\alpha^\omega)_*} \|+ \|(\alpha_{g^{-1}}^\omega(a) - a)\phi^\omega\|_{(M_\alpha^\omega)_*}\\
& \leq \|a\| \, \|\phi \circ \alpha_g - \phi\|_{M_*} + \|\alpha_{g^{-1}}^\omega(a) - a\|_{\phi^\omega}.
\end{align*}
When $g \rightarrow 1_G$, the first summand converges to zero because $\alpha$ is a continuous $G$-action on $M$ and the second summand converges to zero by definition as $a \in M_\alpha^\omega$. 
This shows that $\alpha^\omega$ restricts to a genuine continuous $G$-action on $M_\alpha^\omega$, so the same is true for the restriction of $\alpha_\omega$ to $M_{\omega, \alpha}$. 
\end{proof}

\begin{lemma} \label{lemma:lifting_invariance_compact_sets}
Let $M$ be a von Neumann algebra with a faithful normal state $\phi$ and an action $\alpha\colon G \acts M$. 
Let $z \in M_\alpha^\omega$, $\e >0$, $K \subset G$ a compact set and suppose that $\| \alpha_g^\omega(z)-z\|_{\phi^\omega}^\sharp \leq \e$ for all $g \in K$.
If $(z_n)_{n\in \N}$ is any bounded sequence representing $z$, then 
\[
\lim_{n \rightarrow \omega} \max_{g \in K} \| \alpha_g(z_n)-z_n\|_\phi^\sharp \leq \e.
\]
\end{lemma}
\begin{proof}
Let $\delta>0$.
Then for each $g \in K$ there exists an open neighborhood $g \in U \subset G$ and $W_g \in \omega$ such that
\[\sup_{n \in W_g}\sup_{h \in U}\|\alpha_h(z_n) - \alpha_g(z_n)\|_\phi^\sharp < \delta.\]
Since this obviously yields an open cover of $K$ and $K$ is compact, we can find finitely many elements $g_1, \hdots, g_N \in K$ and an open covering $K \subset \cup_{i=1}^N U_j$ with $g_j \in U_j$ and some $W_1 \in \omega$ such that for $j=1, \hdots, N$ we have
\[\sup_{n \in W_1}\sup_{g \in U_j}\|\alpha_{g}(z_n) - \alpha_{g_j}(z_n)\|_\phi^\sharp < \delta.\]
Since $\max_{g \in K}\| \alpha_g^\omega(z)-z\|_{\phi^\omega}^\sharp \leq \e$, there exists $W_2 \in \omega$ such that for all $n \in W_2$ and $j=1, \hdots, N$ 
\[\|\alpha_{g_j}(z_n) - z_n\|_\phi^\sharp \leq \e+\delta.\]
Hence, for an arbitrary $g \in K$, there is some $j \in \{1, \hdots, N\}$ such that $g \in U_j$ and
\[\| \alpha_g(z_n)-z_n\|_\phi^\sharp \leq \|\alpha_g(z_n) - \alpha_{g_j}(z_n)\|_\phi^\sharp + \|\alpha_{g_j}(z_n) - z_n\|_\phi^\sharp \leq 2\delta+ \e\]
for all $n \in W_1 \cap W_2.$
Since $\delta$ was arbitrary, this proves the claim. 
\end{proof}

\subsection{Cocycle morphisms}

\begin{definition}[cf.\ {\cite[Definition 1.10]{Szabo21cc}}]
Let $\alpha\colon G \acts M$ and $\beta\colon G \acts N$ be two actions of a second-countable locally compact group on von Neumann algebras. 
A (unital) \emph{cocycle morphism} from $(M,\alpha)$ to $(N,\beta)$ is a pair $(\phi,\mathbbm{u})$, where $\phi\colon M \rightarrow N$ is a unital normal $*$-homomorphism and $\mathbbm{u}:G\rightarrow \mathcal{U}(N)$ is a continuous map (in the strong operator topology) such that for all $g,h \in G$ we have
\[
\mathrm{Ad}(\mathbbm{u}_g) \circ \beta_g \circ \phi = \phi \circ \alpha_g \quad\text{and}\quad \mathbbm{u}_g \beta_g(\mathbbm{u}_h) = \mathbbm{u}_{gh}.
\]
If $\mathbbm{u}$ is the trivial map, we identify $\phi$ with $(\phi,1)$ and call $\phi$ equivariant.
\end{definition}

\begin{remark} \label{rem:cocycle-category}
As the arguments in \cite[Subsection 1.3]{Szabo21cc} show, the above endows the class of continuous $G$-actions on von Neumann algebras with a categorical structure, whereby the Hom-sets are given by cocycle morphisms.
The composition is given via
\[
(\psi,\mathbbm{v}) \circ (\phi,\mathbbm{u}) := (\psi \circ \phi, \psi(\mathbbm{u}) \mathbbm{v})
\]
for any pair of cocycle morphisms
\[
(M, \alpha) \overset{(\phi,\mathbbm{u})}{\longrightarrow} (N, \beta) \quad \text{ and } \quad (N, \beta)  \overset{(\psi,\mathbbm{v})}{\longrightarrow} (L, \gamma).
\]
We see furthermore that a cocycle morphism $(\phi, \mathbbm{u})\colon(M, \alpha) \rightarrow (N, \beta)$ is invertible if and only if $\phi$ is a $*$-isomorphism of von Neumann algebras, in which case we have $(\phi, \mathbbm{u})^{-1} = (\phi^{-1}, \phi^{-1}(\mathbbm{u})^*)$.
If this holds, we call $(\phi, \mathbbm{u})$ a \emph{cocycle conjugacy}.
We call two actions $\alpha$ and $\beta$ \emph{cocycle conjugate}, denoted $\alpha \simeq_{\mathrm{cc}} \beta$, if there exists a cocycle conjugacy between them.
\end{remark}

\begin{example} \label{ex:inner-cc}
Let $\alpha\colon G \acts M$ be an action.
Then every unitary $v \in \mathcal{U}(M)$ gives rise to a cocycle conjugacy 
\[
\big( \mathrm{Ad}(v), (v \alpha_g(v)^*)_{g \in G} \big) \colon (M, \alpha) \rightarrow (M, \alpha).\]
We will also write this simply as $\mathrm{Ad}(v)$ when it is clear from context that we are talking about cocycle morphisms. 
When $\beta\colon G \acts N$ is another action and ${(\phi, \mathbbm{u})\colon (M, \alpha) \rightarrow (N, \beta)}$ is a cocycle conjugacy, then
\[ (\phi,\mathbbm{u}) \circ \mathrm{Ad}(v) = \mathrm{Ad}(\phi(v)) \circ (\phi, \mathbbm{u}).\] 
\end{example}

\begin{definition} \label{def:approx-ue}
Let $\alpha:G \acts M$ and $\beta\colon G \acts N$ be two actions on finite von Neumann algebras $M$ and $N$.
Let $\tau_N$ be a faithful normal tracial state on $N$.
Let $(\phi,\mathbbm{u})$ and $(\psi,\mathbbm{v})$ be two cocycle morphisms from $(M,\alpha)$ to $(N,\beta)$. We say that $(\phi,\mathbbm{u})$ and $(\psi,\mathbbm{v})$ are \emph{approximately unitarily equivalent} if there exists a net of unitaries $w_\lambda \in \mathcal{U}(N)$ such that $\|w_\lambda\phi(x)w_\lambda^*-\psi(x)\|_{\tau_N} \to 0$ for all $x\in M$ and $\max_{g\in K} \| w_\lambda \mathbbm{u}_g \beta_g(w_\lambda)^* - \mathbbm{v}_g\|_{\tau_N} \rightarrow 0$ for every compact set $K\subseteq G$.
We denote the relation of approximate unitary equivalence by $\approx_{\mathrm{u}}$.
\end{definition}

\section{One-sided intertwining}
In this section we prove a version of \cite[Lemma~2.1]{Szabo18ssa} (which goes back to \cite[Proposition 2.3.5]{Rordam}) for group actions on semi-finite von Neumann algebras. First we prove the following intermediate lemma:

\begin{lemma}\label{lemma:point_strong_dense_subset}
Let $M, N$ be von Neumann algebras, and let $\tau_N$ be  a faithful, normal, semi-finite trace on $N$. 
Consider a sequence of $*$-homomorphisms $(\theta_n\colon M \rightarrow N)_{n \in \N}$ and a $*$-isomorphism $\theta:M \rightarrow N$ such that $\tau_N \circ \theta = \tau_N \circ \theta_n$ for all $n \in \N$. 
Let $X \subset (M)_1$ be a dense subset in the strong operator topology that contains a sequence of projections $(p_n)_{n \in \N}$ converging strongly to $1_M$ with $\tau_N (\theta(p_n)) < \infty$. 
If $\theta_n (x) \rightarrow \theta(x)$ strongly as $n \rightarrow \infty$ for every $x \in X$, then $\theta_n \rightarrow \theta$ in the point-strong topology as $n \rightarrow \infty$.
\end{lemma}
\begin{proof}
Take $y \in (M)_1$. Since the sequence $(\theta(p_n))_{n \in \N}$ converges strongly to $1_N$, it suffices to show that for all $k \in \N$
\[ (\theta(y)-\theta_n(y)) \theta(p_k) \rightarrow 0 \text{ strongly as } n \rightarrow \infty.\]
Fix $k \in \N$ and $a \in N$ such that $\tau_N(a^*a) < \infty$. Given $\e>0$, there exists $x \in X$ such that 
\[\|\theta(x-y)\theta(p_k)\|_{\tau_N} < \frac{\e}{4\|a\|}.\]
 Then there exists $n_0 \in \N$ such that for all $n \geq n_0$ 
 \[\|(\theta(x p_k) - \theta_n ( x p_k))a\|_{\tau_N} < \frac{\e}{4} \text{ and } \|(\theta(p_k)-\theta_n(p_k))a\|_{\tau_N} < \frac{\e}{4}.\] 
For all $n \geq n_0$ we then get that
\begin{align*}
\|(\theta(y) - \theta_n(y))\theta(p_k)a\|_{\tau_N} &\leq \|\theta(x-y)\theta(p_k)a\|_{\tau_N} + \|\theta_n(x-y)\theta_n(p_k)a\|_{\tau_N}\\
&\quad + \|(\theta(xp_k) - \theta_n(xp_k))a\|_{\tau_N} + \|\theta_n(y)(\theta(p_k) - \theta_n(p_k))a\|_{\tau_N}\\
&< 2\|a\| \|\theta(x-y)\theta(p_k)\|_{\tau_N} +\e/4 + \|(\theta(p_k)-\theta_n(p_k))a\|_{\tau_N} \\
&< \e.
\end{align*}
As $k$ and $a$ were arbitrary, this proves the claim.
\end{proof}

\begin{lemma} \label{lemma:one-sided_intertwining} 
Let $G$ be a second-countable locally compact group.
Let $M$ and $N$ be two von Neumann algebras with separable predual and faithful normal semi-finite traces $\tau_{M}$ and $\tau_{N}$, respectively. 
Let $\alpha\colon G \acts M$ and $\beta\colon G \acts N$ be two actions.
Let $\rho\colon (M, \alpha) \rightarrow (N, \beta)$ be a unital equivariant normal $*$-homomorphism with $\tau_N \circ \rho = \tau_M$.
Suppose there exists a faithful normal state $\phi$ on $N$ and a sequence of unitaries $(w_n)_{n \in \N}$ in $\mathcal{U}(N)$ satisfying
\begin{enumerate}[leftmargin=*,label=$\bullet$]
\item $\mathrm{Ad}(w_n) \circ \rho \to \rho$ in the point-strong topology; 
\item For all $y \in (N)_1$ there exists a sequence $(x_n)_{n \in \N} \subset (M)_1$ such that $y -  w_n\rho(x_n)w_n^* \rightarrow 0$ in the strong operator topology;
\item $\max_{g \in K} \|\beta_g(w_n^*) - w_n^*\|_\phi \rightarrow 0$ for every compact subset $K \subseteq G$.
\end{enumerate}
Then $\rho(\mathcal{Z}(M))=\mathcal{Z}(N)$ and there exists a cocycle conjugacy $(\theta,\mathbbm{v})$ between $\alpha$ and $\beta$ with $\theta|_{\mathcal{Z}(M)}=\rho|_{\mathcal{Z}(M)}$.
In case $\tau_N$ is finite, the existence of such a sequence of unitaries for $\phi = \tau_N$ is equivalent to the condition that $\rho$ is approximately unitarily equivalent to a cocycle conjugacy.
\end{lemma}
\begin{proof}
We note right away that the first two conditions above can always be tested on self-adjoint elements, hence one can equivalently state them with the strong-$*$ topology. 
Denote
\[\mathfrak{m} := \{x \in M \mid \tau_M(x^*x) < \infty\} \subset M.\]
We let $L^2(M,\tau_M)$ denote the GNS-Hilbert space of $M$ with respect to $\tau_M$.
Similarily, we use the notation $L^2(N, \tau_N)$.

Choose a countable subset $X = \{x_n\}_{n \in \N}$ in $(M)_1$ such that $X \cap \mathfrak{m}$ is $\|\cdot\|_{\tau_M}$-dense in $(\mathfrak{m})_1$. 
Take a strongly dense sequence $\{y_n\}_{n \in \N}$ in $(N)_1$.
Choose an increasing sequence of compact subsets $K_n \subseteq G$ such that the union is all of $G$. 

We are going to create a map $\theta: M\to N$ via an inductive procedure. 
For the first step, we choose $x_{1,1} \in (M)_1$ and $z_1 \in \mathcal{U}(N)$ such that
\begin{itemize}
\item $ \|z_1 \rho(x_1) z_1^* - \rho(x_1)\|^\sharp_\phi \leq 1/2$;
\item $\|y_1 - z_1\rho(x_{1,1})z_1^*\|_\phi \leq 1/2$;
\item $\max_{g \in K_1} \|\beta_g(z_1^*) - z_1^*\|_\phi \leq 1/2$.
\end{itemize}
Now assume that after the $n$-th step of the induction we have found $z_1, \hdots, z_n \in \mathcal{U}(N)$ and ${\{x_{l,j}\}_{j \leq l \leq n} \subset (M)_1}$ such that
\begin{enumerate}
\item $\|z_n \rho(x_j) z_n^* - \rho(x_j)\|^\sharp_{\phi \circ \mathrm{Ad}(z_1 \hdots z_{n-1})} \leq 2^{-n}$ for $j= 1, \hdots, n$; \label{eq:commutation_dense}
\item $\|z_n \rho(x_{l,j}) z_n^* - \rho(x_{l,j})\|^\sharp_{\phi \circ \mathrm{Ad}(z_1 \hdots z_{n-1})} \leq 2^{-n}$ for $l=1, \hdots, n-1$ and $j = 1, \hdots, l$; \label{eq:commutation_close_elements}
\item $ \|  z_{n-1}^* \hdots z_1^*  y_j z_1 \hdots z_{n-1} - z_n\rho(x_{n,j})z_n^*\|_{\phi \circ \mathrm{Ad}(z_1 \hdots z_{n-1})} \leq 2^{-n}$ for $j=1, \hdots, n$; \label{eq:close_elements}
\item 
$\max_{g \in K_n} \|\beta_g(z_n^*) - z_n^*\|_{\phi \circ \mathrm{Ad}(z_1 \hdots z_{n-1})} \leq 2^{-n}$ and

$\max_{g \in K_n} \|\beta_g(z_n^*) - z_n^*\|_{\phi \circ \mathrm{Ad}(\beta_g(z_1 \hdots z_{n-1}))} \leq 2^{-n}$.
 \label{eq:invariance}

\label{eq:approximate_fixedness}
\end{enumerate}
Then by our assumptions we can find $z_{n+1} \in \mathcal{U}(N)$ and $\{x_{n+1,j}\}_{j \leq n+1}\subset (M)_1$ such that
\begin{itemize}
\item $\|z_{n+1} \rho(x_j)z_{n+1}^* - \rho(x_j)\|^\sharp_{\phi \circ \mathrm{Ad}(z_1 \hdots z_n)} \leq 2^{-(n+1)}$ for $j= 1, \hdots, n+1$;
\item $\|z_{n+1} \rho(x_{l,j})z_{n+1}^* - \rho(x_{l,j})\|^\sharp_{\phi \circ \mathrm{Ad}(z_1 \hdots z_n)} \leq 2^{-(n+1)}$ for $l=1, \hdots, n$ and $j = 1, \hdots, l$;
\item $ \|z_{n}^* \hdots z_1^*y_jz_1 \hdots z_{n} - z_{n+1}\rho(x_{n+1,j})z_{n+1}^*\|_{\phi \circ \mathrm{Ad}(z_1 \hdots z_n)} \leq 2^{-(n+1)}$ for $j=1, \hdots, n+1$;
\item $\max_{g \in K_{n+1}} \|\beta_g(z_{n+1}^*)- z_{n+1}^*\|_{\phi \circ \mathrm{Ad}(z_1 \hdots z_n)} \leq 2^{-(n+1)}$ and

$\max_{g \in K_{n+1}} \|\beta_g(z_{n+1}^*)- z_{n+1}^*\|_{\phi \circ \mathrm{Ad}(\beta_g(z_1 \hdots z_n))} \leq 2^{-(n+1)}$. 
\end{itemize}
We carry on inductively and obtain a sequence of unitaries $(z_n)_{n \in \N}$ in $\mathcal{U}(N)$ and a family $\{ x_{n,j} \}_{n\in\N, j\leq n}\subset (M)_1$.
For each $n \in \N$, we define $u_n=z_1 \hdots z_n$ and the normal $*$-homomorphism ${\theta_n\colon M \rightarrow N}$ by $\theta_n = \mathrm{Ad}(u_n) \circ \rho$.

For $n > m$ and $j=1, \hdots, m+1$ we get
\begin{align*}
\|\theta_n(x_j) - \theta_m(x_j)\|^\sharp_{\phi} &\leq 
\sum_{k=m}^{n-1} \|\theta_{k+1}(x_j) - \theta_k(x_j)\|^\sharp_\phi\\
&= \sum_{k=m}^{n-1} \|z_{k+1}\rho(x_j)z_{k+1}^* - \rho(x_j)\|^\sharp_{\phi \circ \mathrm{Ad}(z_1 \hdots z_k)}\\
\overset{\ref{eq:commutation_dense}}&{\leq}  \sum_{k=m}^{n-1} 2^{-k-1}.
\end{align*}
We see that for all $j\in \N$ the sequence $(\theta_n(x_j))_{n \in \N}$ is norm-bounded and Cauchy with respect to $\|\cdot\|_\phi^\sharp$.
 This means that it converges to some element in $N$ in the strong-$*$-operator topology.
  A similar calculation using \ref{eq:commutation_close_elements} shows that
for $n > m \geq l \geq j$  
\begin{equation}\label{eq:convergence_close_elements}\|\theta_n(x_{l,j}) - \theta_m(x_{l,j})\|^\sharp_{\phi} < \sum_{k=m}^{n-1} 2^{-k-1}, \end{equation}
so the sequence $(\theta_n(x_{l,j}))_{n \in \N}$ also converges in the strong-$*$-operator topology for all $j \leq l$.
 Since $\theta_n$ is a $*$-homomorphism for all $n \in \N$, we conclude that, restricted to the \cstar-algebra $A\subset M$ generated by $\{x_n\}_{n \in \N} \cup \{x_{l,j}\}_{j \leq l}$, the sequence $(\theta_n)_{n \in \N}$ converges point-$*$-strongly to a $*$-homomorphism $\theta'\colon A \rightarrow N$.
Since $A$ contains a $\|\cdot\|_{\tau_M}$-dense subset of $\mathfrak{m}$, and clearly $\tau_N\circ\theta'=\tau_M|_A$, there is a unique isometry $T\colon L^2(M, \tau_M) \rightarrow L^2(N, \tau_N)$ induced from the formula $T[a]=[\theta'(a)]$ for all $a\in A\cap\mathfrak{m}$.
 Then the normal $*$-homomorphism
\[\theta \colon M \rightarrow N \colon x \mapsto T x T^*\]
 extends $\theta'$ and $\left(\theta_n \big \lvert_\mathfrak{m}\right)_{n \in \N}$ converges point-strongly to $\theta\lvert_\mathfrak{m}$.

We claim that $\theta$ is an isomorphism.
Clearly $\tau_N\circ\theta=\tau_M$ and so $\theta$ is injective.
By applying \ref{eq:close_elements} we find for all $m \geq j$ that
\begin{equation*}
\|\theta_m(x_{m,j}) - y_j\|_\phi = \|z_m\rho(x_{m,j})z_m^* -  z_{m-1}^* \hdots z_1^*  y_j z_1 \hdots z_{m-1}\|_{\phi \circ \mathrm{Ad}(z_1 \hdots z_{m-1})} < 2^{-m}.
\end{equation*}
Combining this with \eqref{eq:convergence_close_elements} for $l=m$ and $n \rightarrow \infty$ we find that
\[\|\theta(x_{m,j}) - y_j\|_\phi \leq \|\theta'(x_{m,j}) - \theta_m(x_{m,j})\|_\phi + \|\theta_m(x_{m,j}) - y_j\|_\phi \leq 2^{-m} + 2^{-m} = 2^{-m+1}.\]
Since the $y_j$ are strongly dense in the unit ball of $N$ and $\theta$ is normal, this implies surjectivity of $\theta$.
By Lemma~\ref{lemma:point_strong_dense_subset} it then follows that $\theta_n\to \theta$ point-strongly as $n\to\infty$.
Since $\theta_n$ is a unitary perturbation of $\rho$ for each $n$, this implies $\rho|_{\mathcal{Z}(M)}=\theta_n|_{\mathcal{Z}(M)}\to\theta|_{\mathcal{Z}(M)}$ and in particular $\rho(\mathcal{Z}(M))=\theta(\mathcal{Z}(M))=\mathcal{Z}(N)$.

For $n > m$ and $g \in K_{m+1}$ we have
\begin{align*}
&\|z_1 \hdots z_n \beta_g(z_n^* \hdots z_1^*) - z_1 \hdots z_m \beta_g(z_m^*\hdots z_1^*)\|_{\phi}^\sharp\\
&\leq \sum_{k=m}^{n-1} \|z_1 \hdots z_k(z_{k+1}\beta_g(z_{k+1}^*) - 1) \beta_g(z_k^*\hdots z_1^*)\|_{\phi}^\sharp\\
&=\sum_{k=m}^{n-1} \big( \|\beta_g(z_{k+1}^*) - z_{k+1}^*\|^2_{\phi \circ \mathrm{Ad}(\beta_g(z_1 \hdots z_k))} + \|\beta_g(z_{k+1}^*) - z_{k+1}^*\|^2_{\phi \circ \mathrm{Ad}(z_1 \hdots z_k)} \big)^{1/2}\\
\overset{\ref{eq:invariance}}&{\leq} \sqrt{2} \sum_{k=m}^{n-1} 2^{-(k+ 1)}.
\end{align*}
From this calculation we see that for every $g \in G$ the sequences $(z_1 \hdots z_n \beta_g(z_n^*\hdots z_1^*))_{n \in \N}$ are Cauchy with respect to $\|\cdot\|^\sharp_\phi$, with uniformity on compact sets. 
It follows that for every $g \in G$, the strong-$*$ limit $\mathbbm{v}_g = \lim_{n \rightarrow \infty} u_n \beta_g(u^*_n)$ exists in $\mathcal{U}(N)$ and that this convergence is uniform (w.r.t.\ $\|\cdot\|^\sharp_\phi$) on compact sets. 
Since $\beta$ is point-strong continuous, this implies the continuity of the assignment $g \mapsto \mathbbm{v}_g$. 

Moreover, for each $g \in G$ and $x \in M$ we have the equalities of limits with respect to the strong operator topology:
\begin{align*}
(\theta \circ \alpha_g) (x) &= \lim_{n \rightarrow \infty} (\mathrm{Ad}(u_n) \circ \rho \circ \alpha_g)(x)\\
&= \lim_{n \rightarrow \infty} (\mathrm{Ad}(u_n) \circ \beta_g \circ \rho) (x)\\
&= \lim_{n \rightarrow \infty} u_n \beta_g(u_n^*)\beta_g(u_n \rho(x) u_n^*) \beta_g(u_n) u_n^*\\
&= (\mathrm{Ad}(\mathbbm{v}_g) \circ \beta_g \circ \theta) (x).
\end{align*}
It follows that $(\theta,\mathbbm{v})$ is a cocycle conjugacy. 

For the last part of the statement, assume that $\tau_N$ is finite.
Then our previous calculations show that in the above situation, $\rho$ is approximately unitarily equivalent to $\theta$.
Conversely, suppose $\rho$ is approximately unitarily equivalent to a cocycle conjugacy $(\theta,\mathbbm{v})$.
In particular, there exists a sequence $(u_n)_{n \in \N} \in \mathcal{U}(N)$ such that $\|u_n\rho(x)u_n^* -  \theta(x)\|_{\tau_N} \rightarrow 0$ for all $x \in M$ and $\|u_n \beta_g(u_n^*) - \mathbbm{v}_g\|_{\tau_N} \rightarrow 0$ uniformly over compact subsets of $G$.
Choose a sequence $\{y_n\}_{n \in \N} \subset (N)_1$ that is strongly dense in $(N)_1$. For all $k,n \in \N$ define $x_{n,k} = \theta^{-1}(u_ny_ku_n^*)$. Then choose an increasing sequence $(m(n))_{n \in \N} \subset \N$ such that
 \begin{equation*}
 \lim_{n \rightarrow \infty} \| \theta(x_{n,k}) - u_{m(n)} \rho (x_{n,k}) u_{m(n)}^*\|_{\tau_N} = 0 \quad \text{for } k \in \N.\end{equation*}
Define $w_n := u_n^*u_{m(n)}$. One can check that these satisfy the assumptions in the lemma. 
\end{proof}

\section{Strongly self-absorbing actions}

\begin{definition}[cf.\ {\cite[Definition 5.1]{Szabo21cc}}] \label{def:strong-absorption}
Let $\alpha:G \acts M$ and $\delta\colon G \acts N$ be two actions of a second-countable locally compact group on finite von Neumann algebras $M$ and $N$ with separable predual.
We say that $\alpha$ \emph{strongly absorbs} $\delta$ if the equivariant embedding
\[\mathrm{id}_M \otimes 1_N\colon (M, \alpha) \rightarrow (M \bar{\otimes} N, \alpha \otimes \delta)\]
is approximately unitarily equivalent to a cocycle conjugacy.
\end{definition}

\begin{definition} \label{def:ssa-action}
Let $G$ be a second-countable locally compact group.
Let $\delta\colon G \acts \mathcal{R}$ be an action on the hyperfinite II$_1$ factor.
We say that $\delta$ is \emph{strongly self-absorbing}, if $\delta$ strongly absorbs $\delta$.
\end{definition}

\begin{definition}
Let $G$ be a second-countable locally compact group.
Let $\alpha\colon G \acts \mathcal{R}$ be an action on the hyperfinite II$_1$ factor.
We say $\alpha$ has \emph{approximately inner half-flip} if the two equivariant embeddings
\[
\mathrm{id}_\mathcal{R} \otimes 1_\mathcal{R}, 1_\mathcal{R} \otimes \mathrm{id}_\mathcal{R}\colon (\mathcal{R}, \alpha) \rightarrow (\mathcal{R} \bar{\otimes} \mathcal{R}, \alpha \otimes \alpha)
\]
are approximately unitarily equivalent (as cocycle morphisms).
\end{definition}

\begin{remark} \label{rem:ssa-misc}
It is well-known that any infinite-dimensional tracial von Neumann algebra $N$ with approximately inner half-flip in the above sense (with $G=\{1\}$) must be isomorphic to $\mathcal R$.
Indeed, it is clear that $N$ must have trivial center, which implies that it is a II$_1$ factor.
Then $N \cong \mathcal{R}$ follows from \cite[Theorem 5.1]{Connes76} under the stronger condition that the flip automorphism on $N\bar{\otimes} N$ is approximately inner, but the weaker condition is seen to be enough via Connes' theorem and the obvious modification of the proof of \cite[Proposition 2.8]{EffrosRosenberg78} that shows the semi-discreteness of $N$.
\end{remark}

\begin{example}\label{example:trivial_action_R}
For any second-countable locally compact group $G$, the trivial action $\mathrm{id}_\mathcal{R}\colon G \acts \mathcal{R}$ has approximately inner half-flip as a consequence of the flip automorphism on a tensor product of matrix algebras $M_n(\C) \otimes M_n(\C)$ being inner. 
It is also seen to be a strongly self-absorbing action.
\end{example}

\begin{remark} \label{rem:compact-fails-A}
Clearly, if a given action $\alpha: G\curvearrowright M$ is cocycle conjugate to $\alpha\otimes\mathrm{id}_{\mathcal R}: G\curvearrowright M\bar{\otimes}\mathcal R$, then the naturality of the crossed product construction implies that $M\rtimes_\alpha G$ is a von Neumann algebra with the McDuff property.
We can use this to argue why the statement in Theorem~\ref{theorem-A} does not extend to the case of actions of topological groups.

Let us pick any ergodic measure-preserving transformation $T: (X,\mu)\to (X,\mu)$ on a standard probability space.
Then the crossed product von Neumann algebra $L^\infty(X)\rtimes_T\mathbb Z$ is well-known to be isomorphic to $\mathcal R$.
If we use this identification to induce a circle action $\alpha: \mathbb{T}\curvearrowright\mathcal R$ from the dual action, then it follows with Takesaki--Takai duality \cite[Chapter X, Theorem 2.3]{Takesaki03} that $\mathcal{R}\rtimes_\alpha\mathbb{T}\cong \mathcal{B}(\ell^2(\mathbb{Z}))\bar{\otimes} L^\infty(X)$.
By what we observed above, it follows that $\alpha$ cannot be cocycle conjugate to $\alpha\otimes\mathrm{id}_{\mathcal R}$.
\end{remark}

\begin{theorem} \label{theorem:sufficient_criterium_strong_absorption}
Let $G$ be a second-countable locally compact group.
Let $\alpha\colon G \acts M$ be an action on a semi-finite von Neumann algebra with separable predual.
Suppose that $\delta\colon G \acts \mathcal{R}$ is an action with approximately inner half-flip such that there exists a unital equivariant $*$-homomorphism $(\mathcal{R}, \delta) \rightarrow (M_{\omega,\alpha}, \alpha_\omega)$.
Then there exists a cocycle conjugacy $(\theta,\mathbbm{v}): (M,\alpha) \to (M\bar{\otimes}\mathcal{R},\alpha \otimes \delta)$ with $\theta|_{\mathcal{Z}(M)}=\operatorname{id}_{\mathcal{Z}(M)}\otimes 1_{\mathcal R}$.
If $M$ is finite, then $\alpha$ strongly absorbs $\delta$.
\end{theorem}
\begin{proof}
Fix a faithful normal state $\phi$ on $M$.
Let $\pi\colon (\mathcal{R},\delta) \rightarrow (M_{\omega,\alpha},\alpha_{\omega})$ be a unital equivariant $*$-homomorphism.
We obtain an induced a map on the algebraic tensor product
\[
\mathcal{R} \odot M \rightarrow M^\omega_\alpha \text{ via } x \otimes m \mapsto \pi(x) m.
\]
Since for each $m \in M_+$ the map $x \mapsto \phi^\omega(\pi(x)m)$ defines a positive tracial functional on $\mathcal{R}$,
 we see that it must be equal to some multiple of the unique tracial state $\tau$ on $\mathcal{R}$ and hence,
  we get for each $x \in \mathcal{R}$ and $m \in M$ that 
\[\phi^\omega(\pi(x) m) = \tau(x) \phi(m) = (\tau \otimes \phi)(x \otimes m).\]
So we see that the map sends the faithful normal state $\tau \otimes \phi$ to $\phi^\omega$ and hence,
 it extends to a unital normal $*$-homomorphism $\mathcal{R} \bar{\otimes} M \rightarrow M^\omega_\alpha$,
 which moreover is $(\delta \otimes \alpha)$-to-$\alpha^\omega$ equivariant.   
In this way we get a unital equivariant normal $*$-homomorphism
\[(\mathcal{R} \bar{\otimes} \mathcal{R} \bar{\otimes}M, \delta \otimes \delta \otimes \alpha) \rightarrow (\mathcal{R} \bar{\otimes} M^\omega_\alpha, \delta \otimes \alpha^\omega),\]
given by $x_1 \otimes x_2 \otimes m \mapsto x_1 \otimes (\phi(x_2) m)$. 
Composing with the canonical inclusion map $\iota\colon \mathcal{R} \bar{\otimes} M^\omega_\alpha \mapsto (\mathcal{R} \bar{\otimes} M)^\omega_{\delta\otimes\alpha}$ we get a unital and equivariant normal $*$-homomorphism
\[\Phi\colon (\mathcal{R} \bar{\otimes} \mathcal{R} \bar{\otimes} M, \delta \otimes \delta \otimes \alpha) \rightarrow ((\mathcal{R} \bar{\otimes} M)^\omega_{\delta \otimes \alpha}, (\delta \otimes \alpha)^\omega)\]
such that
\[
\Phi(x \otimes 1_\mathcal{R} \otimes m) = x \otimes m \text{ for all } x \in \mathcal{R}, m \in M,
\]
and
\[
\Phi(1_\mathcal{R} \otimes \mathcal{R} \otimes M) \subset \iota(1_\mathcal{R} \otimes M^\omega_\alpha).
\]
Since $\delta$ has approximately inner half-flip,
 we can choose a sequence of unitaries $(v_n)_{n \in \N}$ in $\mathcal{R} \bar{\otimes} \mathcal{R}$ such that 
 $\max_{g \in K}\|v_n - (\delta \otimes \delta)_g(v_n)\|_{\tau \otimes \tau} \rightarrow 0$ for all compact subsets $K \subseteq G$ and 
 ${\|x \otimes 1_\mathcal{R} - v_n (1_\mathcal{R} \otimes x)v_n^*\|_{\tau \otimes \tau}  \rightarrow 0}$ for all $x \in \mathcal{R}$. 
 
Define $u_n := \Phi(v_n \otimes 1_M) \subset (\mathcal{R} \bar{\otimes} M)^\omega_{\delta \otimes \alpha}$. This sequence of unitaries satisfies
\begin{itemize} 
\item $[u_n, 1_\mathcal{R} \otimes m] = \Phi([v_n \otimes 1_M, 1_{\mathcal{R} \bar{\otimes} \mathcal{R}} \otimes m]) = 0$ for all $m \in M$;
\item  $\Phi(1_\mathcal{R} \otimes x \otimes m) \in \iota(1_\mathcal{R} \otimes M^\omega_\alpha)$ and
\begin{align*}
\lim_{n \rightarrow \infty} u_n \Phi(1_\mathcal{R} \otimes x \otimes m) u_n^* &=\lim_{n \rightarrow \infty} \Phi((v_n \otimes 1_M)(1_\mathcal{R} \otimes x \otimes m)(v_n^* \otimes 1_M))\\
&=\Phi(x \otimes 1_\mathcal{R} \otimes m) \\&= x \otimes m
\end{align*}
where the limit is taken with respect to the strong operator topology;
\item $\displaystyle \max_{g \in K} \|u_n^* - (\delta \otimes \alpha)^\omega_g(u_n^*)\|_{(\tau \otimes \phi)^\omega} = \max_{g \in K} \|(v_n^* - (\delta \otimes \delta)_g(v_n^*)) \otimes 1_M\|_{\tau \otimes \tau \otimes \phi} \to 0$ for all compact $K \subseteq G$.
\end{itemize}

 Each $u_n$ can be lifted to a sequence of unitaries $(z_n^{(k)})_{k\in \N}$ in $\mathcal{E}_{\delta \otimes \alpha}^\omega \cap \mathcal{N}_\omega(\mathcal{R} \bar{\otimes} M)$.
 Applying a diagonal sequence argument to the $(z_n^{(k)})_{k\in \N}$ and using Lemma~\ref{lemma:lifting_invariance_compact_sets}, we can obtain a sequence of unitaries $(w_n)_{n \in \N}$ in $\mathcal{R} \bar{\otimes} M$ such that
\begin{itemize}
\item $\mathrm{Ad}(w_n)(1_\mathcal{R} \otimes m) - 1_\mathcal{R} \otimes m \rightarrow 0$ strongly for all $m \in M$.
\item $\inf_{m \in (M)_1}\|x - w_n(1_\mathcal{R} \otimes m)w_n^*\|_{\tau \otimes \phi} \rightarrow 0$ for $x \in (\mathcal{R} \bar{\otimes}M)_1$.
\item $\max_{g \in K} \|w_n^* - (\delta \otimes \alpha)_g(w_n^*)\|_{\tau \otimes \phi} \rightarrow 0$ for every compact subset $K \subseteq G$.
\end{itemize}
We conclude that the map $1_\mathcal{R} \otimes \mathrm{id}_M\colon (M, \alpha) \rightarrow (\mathcal{R} \bar{\otimes} M, \delta \otimes \alpha)$ satisfies all the necessary conditions to apply Lemma \ref{lemma:one-sided_intertwining}.
This completes the proof.
\end{proof}

\begin{theorem} \label{theorem:equivalence_ssa}
Let $G$ be a second-countable locally compact group.
Let $\delta:G \acts \mathcal{R}$ be an action on the hyperfinite II$_1$ factor.
Then $\delta$ is strongly self-absorbing if and only if it has approximately inner half-flip and there exists a unital equivariant $*$-homomorphism $(\mathcal{R}, \delta) \rightarrow (\mathcal{R}_{\omega,\delta}, \delta_\omega)$. 
\end{theorem}
\begin{proof}
The `if' direction follows immediately from the previous proposition.
To prove the other direction, we assume that $\delta$ is strongly self-absorbing and reproduce an argument analogous to \cite[Proposition 1.4]{TomsWinter07} and \cite[Proposition 5.5]{Szabo21cc}.
Denote the unique tracial state on $\mathcal{R}$ by $\tau$. Let $(\phi,\mathbbm{u})\colon (\mathcal{R}, \delta) \rightarrow (\mathcal{R} \bar{\otimes} \mathcal{R}, \delta \otimes \delta)$ be a cocycle conjugacy and $u_n \in \mathcal{U}(\mathcal{R} \bar{\otimes} \mathcal{R})$ a sequence of unitaries such that
\begin{equation}\label{eq:approx_coboundary}
\lim_{n \rightarrow \infty} \max_{g \in K} \| u_n(\delta \otimes \delta)_g(u_n^*) - \mathbbm{u}_g\|_{\tau \otimes \tau} = 0 \text{ for every compact } K \subseteq G
\end{equation} 
and
\[
\lim_{n \rightarrow \infty} \|\phi(x) - u_n(x \otimes 1) u_n^*\|_{\tau \otimes \tau} = 0 \text{ for all } x \in \mathcal{R}.
\]
Note that 
\[\mathrm{Ad}(u_n^*) \circ \phi \circ \delta_g = \mathrm{Ad}(u_n^*\mathbbm{u}_g (\delta \otimes \delta)_g(u_n)) \circ (\delta \otimes \delta)_g \circ \mathrm{Ad}(u_n^*) \circ \phi.\]
As a consequence of \eqref{eq:approx_coboundary}, then for every compact $K \subseteq G$ one has 
\begin{equation}\label{eq:approximate_invariance_perturbation}\lim_{n \rightarrow \infty} \max_{g \in K} \sup_{x \in (\mathcal{R})_1} \|(\mathrm{Ad}(u_n^*) \circ \phi \circ \delta_g)(x) - ((\delta \otimes \delta)_g \circ \mathrm{Ad}(u_n^*) \circ \phi)(x)\|_{\tau \otimes \tau} = 0.\end{equation}
In particular, applying this to $x = \phi^{-1}(u_n)$ and using that $(\tau \otimes \tau) = \tau \circ \phi^{-1}$ yields
\[\lim_{n \rightarrow \infty} \max_{g \in K}\|(\mathrm{Ad}(\phi^{-1}(u_n^*)) \circ \delta_g) (\phi^{-1}(u_n)) -(\phi^{-1}\circ (\delta \otimes \delta)_g)(u_n)\|_\tau = 0.\]
Combining this with \eqref{eq:approx_coboundary} again, one gets
\begin{equation}\label{eq:approx_coboundary_inverse}
 \lim_{n \rightarrow \infty} \max_{g \in K} \| \phi^{-1}(u_n^*)\, \delta_g(\phi^{-1}(u_n)) - \phi^{-1}(\mathbbm{u}^*_g)\|_{\tau \otimes \tau} = 0 \text{ for every compact } K \subseteq G.
\end{equation}
First, we prove that $\delta$ has approximately inner half-flip. 
Define the cocyle morphism $(\psi, \mathbbm{v}) := (\phi, \mathbbm{u})^{-1}  \circ (1_\mathcal{R} \otimes \mathrm{id}_\mathcal{R})$. Note that
\begin{align*}
1_\mathcal{R} \otimes \mathrm{id}_\mathcal{R} 
&= (\phi, \mathbbm{u}) \circ (\psi, \mathbbm{v})\\
&\approx_{\mathrm{u}} (\mathrm{id}_\mathcal{R} \otimes 1_\mathcal{R}) \circ (\psi, \mathbbm{v})\\
&= (\psi \otimes 1_\mathcal{R}, \mathbbm{v} \otimes 1).
\end{align*} 
Applying the equivariant flip automorphism to both sides of this equivalence, we get that
\begin{equation}\label{eq:approx_equivalence_first_factor_embedding} \mathrm{id}_\mathcal{R} \otimes 1_\mathcal{R} \approx_\mathrm{u} (1_\mathcal{R} \otimes \psi, 1 \otimes \mathbbm{v}) .\end{equation}
We also get
\begin{align*}
(\psi \otimes 1_\mathcal{R}, \mathbbm{v} \otimes 1) &= (\phi^{-1} \otimes \mathrm{id}_\mathcal{R}, \phi^{-1}(\mathbbm{u})^* \otimes 1) \circ (1_\mathcal{R} \otimes \mathrm{id}_\mathcal{R} \otimes 1_\mathcal{R})\\
\overset{\eqref{eq:approx_equivalence_first_factor_embedding}}&{\approx_{\mathrm{u}}} (\phi^{-1} \otimes \mathrm{id}_\mathcal{R}, \phi^{-1}(\mathbbm{u})^* \otimes 1) \circ (1_\mathcal{R} \otimes 1_\mathcal{R} \otimes \psi, 1\otimes 1 \otimes \mathbbm{v})\\
&=(1_\mathcal{R} \otimes \psi, \phi^{-1}(\mathbbm{u})^* \otimes \mathbbm{v})\\
\overset{\eqref{eq:approx_coboundary_inverse}}&{\approx_{\mathrm{u}}} (1_\mathcal{R} \otimes \psi, 1 \otimes \mathbbm{v}).
\end{align*}
By transitivity we get that ${1_\mathcal{R} \otimes \mathrm{id}_\mathcal{R}\approx_\mathrm{u} \mathrm{id}_\mathcal{R} \otimes 1_\mathcal{R}}$.

Next we prove the existence of a unital equivariant $*$-homomorphism $(\mathcal{R}, \delta) \rightarrow (\mathcal{R}_{\omega,\delta}, \delta_\omega)$. Define the sequence of trace-preserving $*$-homomorphisms 
\[\chi_n = \phi^{-1} \circ \mathrm{Ad}(u_n) \circ (1_\mathcal{R} \otimes \mathrm{id}_\mathcal{R}).\] 
We conclude from \eqref{eq:approximate_invariance_perturbation} that for all $x \in \mathcal{R}$
\[\lim_{n \rightarrow \infty} \max_{g \in K} \|\delta_g(\chi_n(x)) - \chi_n(\delta_g(x))\|_{\tau} =0.\]
From this and the fact that all $\chi_n$ are trace-preserving it also follows that $(\chi_n(x))_{n \in \N}$ belongs to $\mathcal{E}^\omega_\alpha$.
Moreover, for any $x,y \in \mathcal{R}$
\begin{align*}
\lim_{n \rightarrow \infty} \|[x, \chi_n(y)] \|_\tau
&= \lim_{n \rightarrow \infty}\|[\phi(x), u_n(1 \otimes y) u_n^*]\|_{\tau \otimes \tau}\\
&= \lim_{n \rightarrow \infty}\|u_n[x \otimes 1, 1 \otimes y] u_n^*\|_{\tau \otimes \tau}\\
&=0. 
\end{align*}
So the maps $\chi_n$ induce a unital equivariant $*$-homomorphism $(\mathcal{R}, \delta) \rightarrow (\mathcal{R}_{\omega,\delta}, \delta_\omega)$.
\end{proof}

The following can be seen as a direct generalization of the famous McDuff theorem \cite{McDuff70} to actions on semi-finite von Neumann algebras.

\begin{corollary} \label{cor:equivalence_equivariant_McDuff}
Let $G$ be a second-countable locally compact group.
Let $\alpha:G \acts M$ be an action on a semi-finite von Neumann algebra with separable predual and let $\delta:G \acts \mathcal{R}$ be a strongly self-absorbing action on the hyperfinite II$_1$ factor.
Then the following are equivalent:
\begin{enumerate}[leftmargin=*,label=\textup{(\arabic*)}]
\item There exists a cocycle conjugacy $(\theta,\mathbbm{v})\colon (M,\alpha) \to (M\bar{\otimes}\mathcal{R},\alpha \otimes \delta)$ with $\theta|_{\mathcal{Z}(M)}=\operatorname{id}_{\mathcal{Z}(M)}\otimes 1_{\mathcal R}$; \label{prop:McDuff:1}
\item $\alpha \simeq_{\mathrm{cc}} \alpha \otimes \delta$; \label{prop:McDuff:2}
\item There exists a unital equivariant $*$-homomorphism $(\mathcal{R},\delta) \rightarrow (M_{\omega,\alpha}, \alpha_\omega)$. \label{prop:McDuff:3}
\end{enumerate}
\end{corollary}
\begin{proof}
The implication \ref{prop:McDuff:1}$\Rightarrow$\ref{prop:McDuff:2} is tautological.
Since strong self-absorption implies approximately inner half-flip by Proposition~\ref{theorem:equivalence_ssa}, the implication \ref{prop:McDuff:3}$\Rightarrow$\ref{prop:McDuff:1} follows from Proposition~\ref{theorem:sufficient_criterium_strong_absorption}.

In order to prove \ref{prop:McDuff:2}$\Rightarrow$\ref{prop:McDuff:3}, it is enough to show that there exists a unital equivariant $*$-homomorphism $(\mathcal{R}, \delta) \rightarrow ((M \bar{\otimes} \mathcal{R})_{\omega,\alpha \otimes \delta}, (\alpha \otimes \delta)_\omega)$.
We know there exists a unital equivariant $*$-homomorphism $(\mathcal{R}, \delta) \rightarrow (\mathcal{R}_{\omega,\delta}, \delta_\omega)$ by Proposition \ref{theorem:equivalence_ssa}.
Since the latter is unitally and equivariantly contained in $((M \bar{\otimes} \mathcal{R})_{\omega,\alpha \otimes \delta}, (\alpha \otimes \delta)_\omega)$, this finishes the proof.
\end{proof}

The following lemma is a straightforward application of the noncommutative Rokhlin Theorem of Masuda \cite[Theorem 4.8]{Masuda13}.\footnote{This Rokhlin Theorem is actually a variant of Ocneanu's noncommutative Rokhlin Theorem \cite[Theorem 6.1]{Ocneanu85}, and the proof of Masuda's version is essentially the same as Ocneanu's proof.While it is possible to deduce what we need from Ocneanu's Theorem, here we cite Masuda's version for convenience of the reader, as it is directly applicable and there is no need to deal with $\e$-paving families of $G$.}
In the proof we use the following convention: Let $G$ be a discrete amenable group, $\e >0$ and $S \ssubset G$.
We say that $F \ssubset G$ is \emph{$(S,\e)$-invariant} if 
\[
\Big \lvert F \cap \bigcap_{g \in S} g^{-1}F\Big \rvert > (1-\e) |F|.
\]

\begin{lemma}\label{lem:approx-central-embeddings}
Let $\alpha\colon G \acts M$ be an action of a countable discrete group on a McDuff factor with separable predual.
Let $N\subseteq G$ be the normal subgroup consisting of all elements $g\in G$ such that $\alpha_{\omega,g} \in\operatorname{Aut}(M_\omega)$ is trivial.
Suppose that the quotient group $G_0=G/N$ is amenable with quotient map $\pi: G\to G_0$.
Let $\delta: G_0\curvearrowright\mathcal R$ be an action with induced $G$-action $\delta_\pi=\delta\circ\pi$. 
Then there exists an equivariant unital $*$-homomorphism $(\mathcal R,\delta_\pi)\to (M_\omega,\alpha_\omega)$.
\end{lemma}
\begin{proof}
Consider the induced faithful action $\gamma: G_0\curvearrowright M_\omega$ via $\gamma_{gN}=\alpha_{\omega,g}$.
Then clearly the claim is equivalent to finding a $G_0$-equivariant unital $*$-homomorphism $(\mathcal R,\delta)\to (M_\omega,\gamma)$.
Let us introduce some notation.
Let $(x_n)_{n \in \N} \in \ell^\infty(M)$ be a sequence representing an element $X \in M_\omega$. Then we set $\tau_\omega(X) = \lim_{n \rightarrow \omega} x_n$,  where the limit is taken in the $\sigma$-weak topology.
Since $M$ is a factor and $\tau_\omega(X)$ is central, this limit belongs to $\C$. For any $\phi \in M_*$ we have
\[\phi^\omega(X) = \lim_{n \rightarrow \omega}\phi(x_n) = \phi(\tau_\omega(X)) = \tau_\omega(X).\]
In particular, $\tau_\omega$ defines a normal faithful tracial state on $M_\omega$ and we denote $\|X\|_1 = \tau_\omega(|X|)$.

Since $M$ is McDuff we can find a unital $*$-homomorphism $\Phi: \mathcal{R} \rightarrow M_\omega$. Fix $\e >0$ and a symmetric finite subset $S \ssubset G_0$ containing the neutral element.
By \cite[Lemmas 5.6 and 5.7]{Ocneanu85} we are allowed to apply \cite[Theorem 4.8]{Masuda13} to the action $\gamma\colon G_0\curvearrowright M_\omega$.
So if $F \ssubset G_0$ is a finite $(S,\e)$-invariant subset, then there exists a partition of unity of projections $\{E_g\}_{g \in F} \subset M_\omega$ such that
\begin{align}
\sum_{g \in s^{-1}F \cap F} \|\gamma_{s}(E_g) - E_{sg}\|_1 &< 4\e^{1/2} \text{ for all } s \in S; \label{eq:r-lemma1}\\
\sum_{g \in F\setminus s^{-1}F} \|E_g\|_1 &< 3\e^{1/2} \text{ for all } s \in S;\label{eq:r-lemma2}\\
[E_g, \gamma_h(X)] &= 0 \text{ for all } g \in F,\ h\in G_0,\ X \in \Phi(\mathcal{R}).\label{eq:commutation_image}
\end{align}
Define 
\[
\Psi: \mathcal{R} \rightarrow M_\omega \text{ via } \Psi(x) = \sum_{g\in F} \gamma_{g}(\Phi(\delta_{g}^{-1}(x))) E_g.
\]
This is a unital trace-preserving $*$-homomorphism because the projections $E_g$ form a partition of unity and condition \eqref{eq:commutation_image}.
For $s \in S$ and $x \in \mathcal{R}$ we use conditions \eqref{eq:r-lemma1} and \eqref{eq:r-lemma2} to observe
\begin{align*}
\|\gamma_{s}(\Psi(x)) - \Psi(\delta_s(x))\|_1 &=  \Big\| \sum_{g \in F} \gamma_{sg}(\Phi(\delta_g^{-1}(x))) \gamma_{s}(E_g) - \sum_{g \in s^{-1}F} \gamma_{sg}(\Phi(\delta_{g}^{-1}(x))) E_{sg} \Big\|_1
\\
&\leq \sum_{g \in F \cap s^{-1}F}  \left\| \gamma_{sg}(\Phi(\delta_g^{-1}(x))) ( \gamma_{s}(E_g) - E_{sg}) \right\|_1\\
& \qquad + \sum_{g \in F\setminus s^{-1}F} \|\gamma_g(\Phi(\delta_g^{-1}(x))) E_g\|_1 + \sum_{g \in F \setminus sF} \|\gamma_{g}(\Phi(\delta_{s^{-1}g}^{-1}(x)))E_g\|_1 \\
&< 10\e^{1/2}\|x\|.
\end{align*}
Since we can do this for arbitrary $\e >0$ and $S \ssubset G$, the claim follows via a standard reindexing trick.
\end{proof}

The following result recovers a famous result due to Ocneanu \cite[Theorem 1.2 and following remark]{Ocneanu85} as well as his uniqueness theorem of outer actions of amenable groups on $\mathcal R$.
We include this proof for the reader's benefit as it is comparably elementary with the methods established so far.

\begin{theorem} \label{theorem:model-absorption}
Let $G$ and $G_1$ be countable discrete groups with $G_1$ amenable.
Let $\delta: G_1\curvearrowright\mathcal R$ be an outer action and $\alpha\colon G \acts M$ an action on a semi-finite McDuff factor with separable predual.
Then:
\begin{enumerate}[label=\textup{(\roman*)},leftmargin=*]
\item $\delta$ is strongly self-absorbing and cocycle conjugate to any other outer action $G_1\curvearrowright\mathcal R$.\label{theorem:model-absorption:1}
\item Suppose $H\subseteq G$ is a normal subgroup containing all elements $g\in G$ such that $\alpha_{\omega,g}$ is trivial.
Suppose $G_1=G/H$ with quotient map $\pi: G\to G_1$. Consider the induced $G$-action $\delta_\pi = \delta \circ \pi$.
Then $\alpha \simeq_{\mathrm{cc}} \alpha \otimes \delta_\pi$. \label{theorem:model-absorption:2}
\end{enumerate}
\end{theorem}
\begin{proof}
\ref{theorem:model-absorption:1}:
Let $\tau$ be the unique tracial state on $\mathcal R$, which we may use to define the 1-norm $\|\cdot\|_1=\tau(|\cdot|)$ on $\mathcal R$.
Set $\delta^{(2)}=\delta\otimes\delta: G_1\curvearrowright \mathcal R\bar{\otimes}\mathcal R=:\mathcal R^{(2)}$, which is also an outer action.
Since the flip automorphism $\sigma$ on $\mathcal R^{(2)}$ is known to be approximately inner, we may pick a unitary $U\in\mathcal{U}(\mathcal R^{(2)\omega})$ with $UxU^*=\sigma(x)$ for all $x\in\mathcal R^{(2)}$.

By \cite[Theorem 3.2]{Connes77}, the induced action $\delta^{(2)\omega}: G_1\curvearrowright\mathcal R^{(2)}_\omega$ is faithful.
We may hence argue exactly as in the proof of Lemma~\ref{lem:approx-central-embeddings} and apply Masuda's noncommutative Rokhlin lemma.
So let $S\ssubset G_1$ be a symmetric finite set and $\e>0$.
If $F \ssubset G_1$ is a finite $(S,\e)$-invariant subset, then there exists a partition of unity of projections $\{E_g\}_{g \in F} \subset \mathcal R^{(2)}_\omega$ such that
\begin{align}
\sum_{g \in s^{-1}F \cap F} \|\delta^{(2)\omega}_{s}(E_g) - E_{sg}\|_1 &< 4\e^{1/2} \text{ for all } s \in S; \label{eq:r-lemma1-}\\
\sum_{g \in F\setminus s^{-1}F} \|E_g\|_1 &< 3\e^{1/2} \text{ for all } s \in S;\label{eq:r-lemma2-}\\
[E_g, x] &= 0 \text{ for all } g \in F,\ x\in\{ \delta^{(2)\omega}_h(U)\}_{h\in G_1}.\label{eq:commutation_image-}
\end{align}
Define $W = \sum_{g \in F} \delta^{(2)\omega}_{g}(U) E_g$.
This is also a unitary in $\mathcal R^{(2)\omega}$ implementing the flip $\sigma$ because the projections $E_g$ form a partition of unity and condition \eqref{eq:commutation_image-}.
For $s \in S$ we use conditions \eqref{eq:r-lemma1-} and \eqref{eq:r-lemma2-} to observe
\begin{align*}
\|\delta^{(2)\omega}_{s}(W) - W\|_1 &=  \Big\| \sum_{g \in F} \delta^{(2)\omega}_{sg}(U) \delta^{(2)\omega}_{s}(E_g) - \sum_{g \in s^{-1}F} \delta^{(2)\omega}_{sg}(U) E_{sg} \Big\|_1
\\
&\leq \sum_{g \in F \cap s^{-1}F}  \left\| \delta^{(2)\omega}_{sg}(U) ( \delta^{(2)\omega}_{s}(E_g) - E_{sg}) \right\|_1\\ & \qquad + \sum_{g \in F\setminus s^{-1}F} \| \delta^{(2)\omega}_{g}(U) E_g\|_1 + \sum_{g \in F \setminus sF} \|\delta^{(2)\omega}_{g}(U) E_{g}\|_1 \\
&< 10\e^{1/2}.
\end{align*}
Since we can do this for arbitrary $\e >0$ and $S \ssubset G_1$, we can use a reindexing trick to obtain a unitary $W\in\mathcal{U}((\mathcal R^{(2)\omega})^{\delta^{(2)\omega}})$ with $WxW^*=\sigma(x)$ for all $x\in\mathcal R^{(2)}$.
In particular, $\delta$ has approximately inner half-flip.
If we apply Lemma~\ref{lem:approx-central-embeddings} for $G=G_1$, $N=\{1\}$ and $\delta$ in place of $\alpha$, it follows with Theorem~\ref{theorem:equivalence_ssa} that $\delta$ is strongly self-absorbing.
If $\gamma\colon G_1\curvearrowright\mathcal R$ is another outer action, then the same follows for $\gamma$.
By applying Lemma~\ref{lem:approx-central-embeddings} and Corollary~\ref{cor:equivalence_equivariant_McDuff} twice, we obtain that $\gamma$ and $\delta$ absorb each other, hence they are cocycle conjugate.

\ref{theorem:model-absorption:2}:
Define $N$ to be the subgroup of all elements $g\in G$ such that $\alpha_{\omega,g}$ is trivial, and set $G_0=G/N$ with quotient map $\pi^0: G\to G_0$.
By assumption we have $N\subseteq H$, hence $G_1$ can be viewed as a quotient of $G_0$ via a map $\pi^{0\to 1}: G_0\to G_1$.
Then $\pi=\pi^{0\to 1}\circ\pi^0$ and the action $\delta_{\pi^{0\to 1}}:=\delta\circ\pi^{0\to 1}$ is a $G_0$-action with $(\delta_{\pi^{0\to 1}})_{\pi^0}=\delta_\pi$.
By Lemma~\ref{lem:approx-central-embeddings}, it follows that there exists an equivariant unital $*$-homomorphism $(\mathcal R,\delta_\pi) \to (M_\omega,\alpha_\omega)$.
Since $\delta$ was strongly self-absorbing, so is $\delta_\pi$ as a $G$-action and the claim follows by Corollary~\ref{cor:equivalence_equivariant_McDuff}.
\end{proof}

\begin{remark}
We note that the factor $M$ in Theorem~\ref{theorem:model-absorption} is only assumed to be semi-finite because at this point in the article, we have only proved the absorption theorem in this setting.
Upon having access to Theorem~\ref{thm:general-McDuff}, the same proof verbatim allows one to remove the assumption that $M$ needs to be semi-finite.
\end{remark}

\section{Actions of discrete amenable groupoids}

We begin by recalling the definition of a discrete measured groupoid. This concept dates back to \cite{Mackey63}. 

\begin{definition}A discrete measured groupoid $\mathcal{G}$ is a groupoid in the usual sense that carries the following additional structure:
\begin{enumerate}[label=$\bullet$,leftmargin=*]
\item The groupoid $\mathcal{G}$ is a standard Borel space and the units $\mathcal{G}^{(0)} \subset \mathcal{G}$ form a Borel subset.
\item The source and target maps $s,t\colon \mathcal{G} \rightarrow \mathcal{G}^{(0)}$ are Borel and countable-to-one.
\item Define $\mathcal{G}^{(2)}:= \{(g,h) \in \mathcal{G} \times \mathcal{G} \mid s(g) = t(h)\}.$ The multiplication map $\mathcal{G}^{(2)} \rightarrow \mathcal{G}\colon (g,h) \mapsto gh$ and the inverse map $\mathcal{G}\rightarrow \mathcal{G}\colon g \mapsto g^{-1}$ are Borel.
\item $\mathcal{G}^{(0)}$ is equipped with a measure $\mu$ satisfying the following property.
Let $\mu_s$ and $\mu_t$ denote the $\sigma$-finite measures on $\mathcal{G}$ obtained by integrating the counting measure over $s,t\colon \mathcal{G}\rightarrow \mathcal{G}^{(0)}$, respectively. Then $\mu_s \sim \mu_t$. 
\end{enumerate}
\end{definition}
\begin{example}
An important example of a discrete measured groupoid is the \emph{transformation groupoid} associated to a non-singular action $G \acts (X, \mu)$ of a countable, discrete group $G$ on a standard measure space $(X, \mu)$. In that case the unit space can be identified with $X$ and the measure $\mu$ satisfies the necessary requirements. We denote this transformation groupoid by $G \ltimes X$. 
\end{example}
We assume the reader is familiar with the concept of amenability for discrete measured groupoids, see \cite[Definition 3.2.8]{AnantharamanRenault00}. In particular, recall that a groupoid $\mathcal{G}$ is amenable if and only if the associated equivalence relation
\[
\big\{ \big( s(g),t(g) \big) \mid g \in \mathcal{G} \big\}
\]
and almost all associated isotropy groups
\[\{g \in \mathcal{G} \mid s(g) = t(g) = x\} \quad \text{for } x \in \mathcal{G}^{(0)}\]
are amenable (see e.g.\ \cite[Corollary 5.3.33]{AnantharamanRenault00}). 
In case of a non-singular action $G \acts (X, \mu)$, the associated transformation groupoid $G \ltimes X$ is amenable if and only if the action is amenable in the sense of Zimmer (\cite{Zimmer78, Zimmer77, Zimmer77b}).

\begin{remark}
In this paper we work with measurable fields of all kinds of separable structures, such as Polish spaces, Polish groups, von Neumann algebras with separable predual, and fields that can be derived from these.
For Polish groups the definition is explicitly given in \cite{Sutherland85}, while the other notions can be defined in an analogous way.
We only consider the measurable setting and hence will often implicitly discard sets of measure zero whenever needed.
This means all measurable fields, groupoids and isomorphisms between measure spaces are defined up to sets of measure zero.
Because of this, all statements should be interpreted as holding only almost everywhere whenever appropriate.
This also means that we have no problem to apply the von Neumann measurable selection theorem (see e.g.\ \cite[Theorem 18.1]{Kechris95}) to obtain measurable sections after deletion of a suitable null set, and we will often omit the fine details related to such arguments.  
\end{remark}

\begin{definition}
Let $\mathcal{G}$ be a discrete measured groupoid with unit space $(X,\mu)$. An \emph{action} $\alpha$ of $\mathcal{G}$ on a measurable field $(B_x)_{x \in X}$ of factors with separable predual is given by a measurable field of $*$-isomorphisms \[\mathcal{G} \ni g \mapsto \alpha_g\colon B_{s(g)} \rightarrow B_{t(g)},\] 
satisfying $\alpha_g \circ \alpha_h = \alpha_{gh}$ for all $(g,h) \in \mathcal{G}^{(2)}$.
\end{definition}

\begin{definition}\label{def:cc_groupoid_actions}
Let $\mathcal{G}$ be a discrete measured groupoid with unit space $(X,\mu)$. Suppose that $\alpha$ and $\beta$ are actions of $\mathcal{G}$ on the measurable fields of factors with separable predual $(B_x)_{x \in X}$ and $(D_x)_{x \in X}$, respectively. The actions are said to be \emph{cocycle conjugate} if there exists a measurable field of $*$-isomorphisms $X \ni x \mapsto \theta_x\colon B_x \rightarrow D_x$ and a measurable field of unitaries $\mathcal{G} \ni g \mapsto w_g \in \mathcal{U}(D_{t(g)})$ satisfying  
\begin{align*}
\theta_{t(g)} \circ \alpha_g \circ \theta_{s(g)}^{-1} &= \mathrm{Ad} w_g \circ \beta_g \text{ for all } g \in \mathcal{G}\\
  w_g \beta_g(w_h) &= w_{gh} \text{ for all } (g,h) \in \mathcal{G}^{(2)}.
\end{align*}
\end{definition}

\begin{example} \label{ex:central_decomposition}
Let $B$ be a von Neumann algebra acting on a separable Hilbert space $\mathcal{H}$. Then we can centrally decompose $B$ as 
\[ (B, \mathcal{H}) = \int_{X}^\oplus (B_x, \mathcal{H}_x)\, d\mu(x), \]
where $(X,\mu)$ is a standard probability space such that $L^\infty(X,\mu) \cong \mathcal{Z}(B)$ (see e.g.\ \cite[Theorem IV.8.21]{Takesaki02}).
In this way we get a measurable field of factors $(B_x)_{x \in X}$. When $B$ is of type I, II$_1$, II$_\infty$ or III, every $B_x$ has the same type by \cite[Corollary V.6.7]{Takesaki02}.
We claim that if $B \cong B \bar{\otimes} \mathcal{R}$, then every fibre $B_x$ is McDuff.
Pick a $*$-isomorphism $\Phi\colon B \rightarrow B \bar{\otimes} \mathcal{R}$.
Then there exists (see for example \cite[Theorem III.2.2.8]{Blackadar}) a unitary $U\colon \mathcal{H} \otimes \ell^2(\N) \rightarrow \mathcal{H} \otimes L^2(\mathcal{R}, \tau_\mathcal{R}) \otimes \ell^2(\N)$ such that the amplification of $\Phi$ is spatial, i.e.\ $\Phi(b) \otimes 1 = U(x \otimes 1)U^*$.
We have the decompositions
\[(B \otimes \C, \mathcal{H} \otimes \ell^2(\N)) =\int_X^\oplus (B_x \otimes \C, \mathcal{H}_x \otimes \ell^2(\N)) \, d \mu(x),\text{ and}\]
\[(B \bar{\otimes} \mathcal{R} \otimes \C,\mathcal{H} \otimes L^2(\mathcal{R}, \tau_\mathcal{R}) \otimes \ell^2(\N))  = \int_X^\oplus \left(B_x \bar{\otimes}\mathcal{R} \otimes \C, \mathcal{H}_x \otimes L^2(\mathcal{R}, \tau_\mathcal{R}) \otimes \ell^2(\N)\right) \, d\mu(x).\]
As the amplification of $\Phi$ necessarily maps the diagonal algebras (i.e.\ the respective centers) to each other, we can use the fact that the disintegration is unique \cite[Theorem 8.23]{Takesaki02}. In particular, this means every $B_x$ is isomorphic to some $B_y \bar{\otimes} \mathcal{R}$ and hence, $B_x \cong B_x \bar{\otimes} \mathcal{R}$. 

Now suppose $\alpha\colon G \acts B$ is an action of a countable discrete group. 
Let ${\mathcal{G} = G \ltimes X}$ denote the transformation groupoid associated to the action on $(X, \mu)$ induced by $\alpha$.
 Then $\alpha$ can be disintegrated as an action $\bar{\alpha}$ of $\mathcal{G}$ on the measurable field $(B_x)_{x \in X}$ (see e.g.\ \cite[Corollary X.3.12]{Takesaki02}\footnote{When the field of factors $(B_x)_{x \in X}$ is constant (for example when $B$ is injective type II$_1$ and all $B_x$ are $\mathcal{R}$), this construction dates back to \cite{SutherlandTakesaki89}. 
 There, the groupoid $\mathcal{G}$ and action $\bar{\alpha}$ are also called the \emph{ancillary groupoid} and \emph{ancillary action} associated to $\alpha$. }) such that given $b= \int_x^\oplus b_x\,  d \mu(x)$, we have
\[\alpha_g(b)_{g \cdot x} = \bar{\alpha}_{(g,x)}(b_{x}) \text{ for } (g,x) \in \mathcal{G}.\]

Assume $\beta\colon G \acts D$ is another action on a separably acting von Neumann algebra $(D, \mathcal{K}) = \int_X^\oplus (D_x, \mathcal{K}_x)\, d\mu(x)$, and assume that $\beta$ induces the same action on $(X,\mu)$ as $\alpha$. 
Let $\bar{\beta}$ denote its decomposition as an action of $\mathcal{G}$ on $(D_x)_{x \in X}$. 
If $\bar{\alpha}$ and $\bar{\beta}$ are cocycle conjugate in the sense of Definition~\ref{def:cc_groupoid_actions}, then $\alpha$ and $\beta$ are cocycle conjugate as actions on von Neumann algebras.
Indeed, let $X \ni x \mapsto \theta_x\colon A_x \rightarrow B_x$ and $\mathcal{G} \ni (g,x) \mapsto w_{(g,x)} \in \mathcal{U}(B_{g \cdot x})$ denote the measurable fields of $*$-isomorphisms and unitaries realizing a cocycle conjugacy between $\bar{\alpha}$ and $\bar{\beta}$.
This gives rise to a $*$-isomorphism $\theta\colon A \rightarrow B$ given by $\theta(a)_x = \theta_x(a_x)$ for $a = \int_X^\oplus a_x\, \mathrm{d} \mu(x) \in A$, and for each $g \in G$ we get a unitary $\mathbbm{v}_g \in \mathcal{U}(B)$ by
$(\mathbbm{v}_g)_x = w_{(g, g^{-1}\cdot x)}$. The pair $(\theta, \mathbbm{v})$ is a cocycle conjugacy.

Conversely, one can show that every cocycle conjugacy $(\theta, \mathbbm{v}):(B, \alpha) \rightarrow (D, \beta)$ with $\theta\big\lvert_{L^\infty(X)} =  \mathrm{id}\big\lvert_{L^\infty(X)}$ gives rise to a cocycle conjugacy in the sense of Definition~\ref{def:cc_groupoid_actions}.
\end{example}

We will subsequently need the following lemma (albeit only for discrete groups) about strongly self-absorbing actions.

\begin{lemma} \label{lem:special-cc}
Let $G_j$, $j=1,2$, be two second-countable locally compact groups with a continuous group isomorphism $\phi: G_1\to G_2$.
Let $\delta^{(j)}\colon G_j\curvearrowright\mathcal R$, $j=1,2$, be two strongly self-absorbing actions and choose a cocycle conjugacy $(\Phi,\mathbb{U})\colon (\mathcal R,\delta^{(2)})\to (\mathcal R\bar{\otimes}\mathcal R,\delta^{(2)}\otimes\delta^{(2)})$ that is approximately unitarily equivalent to $\operatorname{id}_{\mathcal R}\otimes 1_{\mathcal R}$.
(Note that $\phi$ allows us to identify $(\Phi,\mathbb{U}\circ\phi)$ with a cocycle conjugacy between the $G_1$-action $\delta^{(2)}\circ\phi$ and its tensor square.)
Let $\alpha^{(j)}\colon  G_j\curvearrowright M_j$, $j=1,2$, be two actions on separably acting von Neumann algebras.
Given a cocycle conjugacy
\[
(\theta,\mathbbm{v})\colon (M_1,\alpha^{(1)})\to \big( M_2\bar{\otimes}\mathcal{R}, (\alpha^{(2)}\otimes\delta^{(2)})\circ\phi \big),
\]
and a conjugacy $\Delta\colon (\mathcal R, \delta^{(2)} \circ \phi)\to(\mathcal R,\delta^{(1)})$,
consider the cocycle conjugacy of $G_1$-actions
\[
(\Psi,\mathbb{V})= \big( (\theta, \mathbbm{v})^{-1} \otimes \Delta \big) \circ \big( \mathrm{id}_{M_2} \otimes (\Phi, \mathbbm{U}\circ\phi) \big) \circ (\theta, \mathbbm{v})
\]
between $(M_1,\alpha^{(1)})$ and $(M_1\bar{\otimes}\mathcal R,\alpha^{(1)}\otimes \delta^{(1)} )$.
Then there exists a sequence of unitaries $y_n\in\mathcal{U}(M_1\otimes\mathcal R)$ such that
\[
\operatorname{Ad}(y_n)\circ (\operatorname{id}_{M_1}\otimes 1_{\mathcal R})\to \Psi,\quad \operatorname{Ad}(y_n^*)\circ\Psi \to \operatorname{id}_{M_1}\otimes 1_{\mathcal R}
\]
point-strongly, and such that
\[
y_n(\alpha^{(1)}\otimes\delta^{(1)})_g(y_n)^* \to \mathbb{V}_g,\quad y_n^*\mathbb{V}_g(\alpha^{(1)}\otimes\delta^{(1)})_g(y_n)\to 1_{M_1\otimes\mathcal R}
\]
in the strong-$*$ operator topology for all $g\in G$ and uniformly over compact sets.
\end{lemma}
\begin{proof}
By assumption, there is a sequence of unitaries $z_n\in\mathcal{U}(\mathcal{R}\bar{\otimes}\mathcal{R})$ such that 
\begin{equation} \label{eq:main-tech:ssa1}
\|\Phi(x) - z_n(x\otimes 1_{\mathcal R})z_n^*\|_2\to 0\quad\text{for all } x\in\mathcal R.
\end{equation}
and 
\begin{equation} \label{eq:main-tech:ssa2}
\max_{h\in K} \|\mathbb{U}_h - z_n(\delta^{(2)}\otimes\delta^{(2)})_h(z_n)^*\|_2\to 0 \quad\text{for every compact } K\subseteq G_2.
\end{equation}
By definition, we have
\[
\Psi = (\theta^{-1}\otimes\Delta)\circ (\operatorname{id}_{M_2}\otimes\Phi)\circ\theta
\]
and
\[
\mathbb{V}_g = ({\theta}^{-1}\otimes\Delta)\Big( (\operatorname{id}_{M_2}\otimes\Phi)(\mathbbm{v}_g) \cdot (1_{M_2}\otimes \mathbb{U}_{\phi(g)})\cdot (\mathbbm{v}_g^* \otimes 1_{\mathcal R}) \Big),\quad g\in G_1.
\]
If we consider the sequence of unitaries 
\[
y_n=(\theta^{-1}\otimes\Delta)(1_{M_2}\otimes z_n),
\]
then we can observe with \eqref{eq:main-tech:ssa1} that
\[
\operatorname{Ad}(y_n^*)\circ\Psi \to (\theta^{-1}\otimes\Delta)\circ (\operatorname{id}_{M_2}\otimes\operatorname{id}_{\mathcal R}\otimes 1_{\mathcal R})\circ\theta = \operatorname{id}_{M_1}\otimes 1_{\mathcal R}
\]
as well as
\[
\operatorname{Ad}(y_n)\circ(\operatorname{id}_{M_1}\otimes 1_{\mathcal R}) = \operatorname{Ad}(y_n)\circ(\theta^{-1}\otimes\Delta)\circ (\operatorname{id}_{M_2}\otimes\operatorname{id}_{\mathcal R}\otimes 1_{\mathcal R})\circ\theta \to \Psi
\]
point-strongly.
Moreover, given $g\in G_1$, the fact that $(\theta^{-1}, \big( \theta^{-1}(\mathbbm{v}^*_g) \big)_{g\in G_1} )$ is the inverse of $(\theta,\mathbbm{v})$ leads to the equation $\alpha^{(1)}_g\circ\theta^{-1}=\theta^{-1}\circ\operatorname{Ad}(\mathbbm{v}_g)\circ(\alpha^{(2)}\otimes\delta^{(2)})_{\phi(g)}$.
If we combine this with \eqref{eq:main-tech:ssa1} and \eqref{eq:main-tech:ssa2}, we can see that
\[
\begin{array}{cl}
\multicolumn{2}{l}{ y_n^*\mathbb{V}_g(\alpha^{(1)}\otimes\delta^{(1)})_g(y_n) } \\
=& y_n^*\mathbb{V}_g\cdot (\theta^{-1}\otimes\Delta)\big( \operatorname{Ad}(\mathbbm{v}_g\otimes 1_{\mathcal R})(1_{M_2}\otimes(\delta^{(2)}\otimes\delta^{(2)})_{\phi(g)}(z_n)) \big)  \\
=& (\theta^{-1}\otimes\Delta)\Big( (1_{M_2}\otimes z_n)^*\cdot (\operatorname{id}_{M_2}\otimes\Phi)(\mathbbm{v}_g) \cdot (1_{M_2}\otimes \mathbb{U}_{\phi(g)}) \cdots \\
& \phantom{ (\theta^{-1}\otimes\operatorname{id}_{\mathcal{R}})\Big( }\cdots (1_{M_2}\otimes(\delta^{(2)}\otimes\delta^{(2)})_{\phi(g)}(z_n))\cdot (\mathbbm{v}_g^* \otimes 1_{\mathcal R}) \Big) \\
=& (\theta^{-1}\otimes\Delta)\Big( \underbrace{ \big( \operatorname{id}_{M_2}\otimes (  \operatorname{Ad}(z_n^*)\circ\Phi ) \big) (\mathbbm{v}_g) }_{\to \mathbbm{v_g}\otimes 1_{\mathcal R} } 
\cdot \big( 1_{M_2}\otimes \underbrace{ z_n^*\mathbb{U}_{\phi(g)}(\delta^{(2)}\otimes\delta^{(2)})_{\phi(g)}(z_n) }_{\to 1_{\mathcal{R}\otimes\mathcal{R}} } \big)\cdot (\mathbbm{v}_g^* \otimes 1_{\mathcal R}) \Big) \\
\to & 1_{M_1\otimes\mathcal R} 
\end{array}
\]
in the strong-$*$ operator topology, uniformly over compact subsets.
Analogously, we observe the convergence
\[
\begin{array}{cl}
\multicolumn{2}{l}{ y_n(\alpha^{(1)}\otimes\delta^{(1)})_g(y_n)^* }\\
=& (\theta^{-1}\otimes\Delta)\Big( (1_{M_2}\otimes z_n) \cdot \operatorname{Ad}(\mathbbm{v}_g\otimes 1_{\mathcal R})(1_{M_2}\otimes(\delta^{(2)}\otimes\delta^{(2)})_{\phi(g)}(z_n^*) ) \Big) \\
=& (\theta^{-1}\otimes\Delta)\Big( \underbrace{ \operatorname{Ad}(1_{M_2}\otimes z_n)(\mathbbm{v}_g\otimes 1_{\mathcal R})}_{\to (\operatorname{id}_{M_2}\otimes\Phi)(\mathbbm{v}_g)}
 \cdot \big( 1_{M_2}\otimes \underbrace{ z_n(\delta^{(2)}\otimes\delta^{(2)})_{\phi(g)}(z_n^*) }_{\to \mathbb{U}_{\phi(g)} } \big) \cdot (\mathbbm{v}_g^*\otimes 1_{\mathcal R}) \Big) \\
\to & \mathbb{V}_g
\end{array}
\]
uniformly over compact sets.
This finishes the proof.
\end{proof}

In the proof of the main technical result of this section we will make use of the following variant, due to Popa--Shlyakhtenko--Vaes, of the cohomology lemmas in \cite[Appendix]{JonesTakesaki84} and \cite[Theorem 5.5]{Sutherland85}.

\begin{lemma}[{\cite[Theorem 3.5]{PopaShlyakhtenkoVaes20}}] \label{lemma:cohomology_lemma}
Let $\mathcal{S}$ be an amenable countable nonsingular equivalence relation on the standard probability space $(X,\mu)$.
Let $(G_x \acts P_x)_{x \in X}$ be a measurable field of continuous actions of Polish groups on Polish spaces, on which $\mathcal{S}$ is acting by conjugacies: we have measurable fields of group isomorphisms ${\mathcal{S} \ni (x,y) \mapsto \gamma_{(x,y)}\colon G_y \mapsto G_x}$ and homeomorphisms $\mathcal{S} \ni (x,y) \mapsto \beta_{(x,y)}\colon P_y \mapsto P_x$ satisfying 
\[
\gamma_{(x,y)} \circ \gamma_{(y,z)} = \gamma_{(x,z)}, \quad \beta_{(x,y)} \circ \beta_{(y,z)} = \beta_{(x,z)}, \quad \beta_{(x,y)} (g \cdot \pi) = \gamma_{(x,y)}(g) \cdot \beta_{(x,y)}(\pi)
\]
for all $(x,y), (y,z) \in \mathcal{S}$ and $g \in G_y, \pi \in P_y$.
Let $X \ni x \mapsto \sigma'(x) \in P_x$ be a measurable section. Assume that for all $(x,y) \in \mathcal{S}$, the element $\sigma'(x)$ belongs to the closure of ${G_x \cdot \beta_{(x,y)}(\sigma'(y))}$. 
Then there exists a measurable family $\mathcal{S} \ni (x,y) \mapsto v(x,y) \in G_x$ and a section $X \ni x \mapsto \sigma(x) \in P_x$ satisfying:
\begin{itemize}
\item $v$ is a 1-cocycle: $v(x,y) \gamma_{(x,y)}(v(y,z)) = v(x,z)$ for all $(x,y), (y,z)\in \mathcal{S}$;
\item $v(x,y)\cdot \beta_{(x,y)}(\sigma(y)) = \sigma(x)$ for all $(x,y) \in \mathcal{S}$.
\end{itemize}
\end{lemma}

\begin{remark} \label{rem:Polish-spaces}
Before we state and prove our main technical result in detail, we would like to outline for what kind of input data the lemma above will be used.
In the situation considered below, the typical Polish space $P$ will be the space of cocycle conjugacies $(M,\alpha)\to (N,\beta)$, where $\alpha: H\curvearrowright M$ and $\beta: H\curvearrowright N$ are actions of a countable discrete group $H$ on separably acting von Neumann algebras.
Here we consider the topology defined by declaring that one has a convergence of nets $(\theta^{(\lambda)},\mathbbm{v}^{(\lambda)})\to(\theta,\mathbbm{v})$ if and only if $\mathbbm{v}^{(\lambda)}_g\to\mathbbm{v}_g$ in the strong-$*$ operator topology for all $g\in H$, and $\|\varphi\circ\theta-\varphi\circ\theta^{(\lambda)}\|\to 0$ for all $\varphi\in N_*$.
This generalizes the well-known topology on the space of isomorphisms $M\to N$ that one usually called the ``$u$-topology''; cf.\ \cite{Haagerup75, Winslow98}.
The typical Polish group acting on this Polish space would be the unitary group $\mathcal{U}(N)$, which is equipped with the strong-$*$ operator topology, where the action is defined via composition with the inner cocycle conjugacy as per Example~\ref{ex:inner-cc} and Remark~\ref{rem:cocycle-category}.
In other words, a unitary $w\in\mathcal{U}(N)$ moves the cocycle conjugacy $(\theta,\mathbbm{v})$ to $\operatorname{Ad}(w)\circ(\theta,\mathbbm{v}) = \big( \operatorname{Ad}(w)\circ\theta, (w\mathbbm{v}\beta_g(w)^*)_{g\in H} \big)$.

If we assume in addition that $M$ and $N$ are semi-finite, we may pick a faithful normal semi-finite tracial weight $\tau$ on $N$.
Assume that $(\Psi,\mathbb{V})\in P$ is a cocycle conjugacy.
Then it follows from \cite[Proposition 3.7]{Haagerup75} that on the space of all isomorphisms $\Psi': M\to N$ with $\tau\circ\Psi'=\tau\circ\Psi$, the $u$-topology coincides with the topology of point-strong convergence.
As a direct consequence, we may conclude the following. 
If $(\Phi,\mathbb{U})\in P$ is another cocycle conjugacy and there exists a net of unitaries $w_\lambda\in\mathcal{U}(N)$ such that $w_\lambda\mathbb{V}_g\beta_g(w_\lambda)^*\to \mathbb{U}_g$ for all $g\in H$ in the strong-$*$ operator topology and $\operatorname{Ad}(w_\lambda)\circ\Psi\to\Phi$ point-strongly, then $(\Phi,\mathbb{U})\in\overline{ \mathcal{U}(N)\cdot (\Psi,\mathbb{V}) }$.
\end{remark}

The following can be seen as the main technical result of this section, which we previously referred to as a kind of measurable local-to-global principle.

\begin{theorem} \label{thm:main_technical}
Let $G \acts (X,\mu)$ be an amenable action (in the sense of Zimmer) of a countable, discrete group on a standard probability space.
Let $\alpha$ be an action of ${\mathcal{G}:= G \ltimes X}$ on a measurable field of semi-finite factors with separable predual $(B_x)_{x \in X}$. 
Denote by ${X \ni x \mapsto H_x}$ the measurable field of isotropy groups. 
For any action ${\delta\colon G \acts \mathcal{R}}$ on the hyperfinite II$_1$ factor, we define a tensor product action $\alpha \otimes \delta$ of $\mathcal{G}$ on the field of factors $(B_x \bar{\otimes} \mathcal{R})_{x \in X}$ by $(\alpha \otimes \delta)_{(g,x)} = \alpha_{(g,x)} \otimes \delta_g$.
If $\delta$ is strongly self-absorbing, then the following are equivalent:
 \begin{enumerate}[leftmargin=*,label=\textup{(\arabic*)}]
 \item $\alpha \simeq_{\mathrm{cc}} \alpha \otimes \delta$; \label{prop:main_technical:1}
 \item For almost every $x \in X$ we have $\alpha{|_{H_x}} \simeq_{\mathrm{cc}} (\alpha \otimes \delta){|_{H_x}}$ as actions of $H_x$ on $B_x$ and $B_x \bar{\otimes} \mathcal{R}$. \label{prop:main_technical:2}
 \end{enumerate}
\end{theorem}
\begin{proof}
We note that by following the argument outlined in Example~\ref{ex:central_decomposition} and by applying Corollary~\ref{cor:equivalence_equivariant_McDuff}, we see that $\alpha \simeq_{\mathrm{cc}} \alpha \otimes \delta$ implies that one can find a cocycle conjugacy between $\alpha$ and $\alpha\otimes\delta$ that induces the identity map on $X$.
Hence \ref{prop:main_technical:1} implies \ref{prop:main_technical:2}.

In order to prove the other implication, assume \ref{prop:main_technical:2} holds. 
To verify \ref{prop:main_technical:1}, we will show the existence of a measurable field of $*$-isomorphisms ${X \ni x \mapsto \theta_x \colon B_x \rightarrow B_x \bar{\otimes} \mathcal{R}}$ and unitaries ${\mathcal{G} \ni (g,x) \mapsto w_{(g,x)} \in \mathcal{U}(B_{g \cdot x} \bar{\otimes} \mathcal{R})}$ such that
\begin{align}
\theta_{g \cdot x} \circ \alpha_{(g,x)} \circ \theta_{x}^{-1} &= \mathrm{Ad} \big( w_{(g,x)} \big) \circ (\alpha \otimes \delta)_{(g,x)} \text{ for all } (g,x) \in \mathcal{G}\label{eq:required_cc}\\
  w_{(g,h \cdot x)} (\alpha \otimes \delta)_{(g,h \cdot x)}(w_{(h,x)}) &= w_{(gh,x)} \text{ for all } g,h \in G, x \in X.\label{eq:required_cocycle}
\end{align}
For every $x \in X$, denote by $P_x$ the Polish space of cocycle conjugacies from $(B_x, \alpha|_{H_x})$ to $(B_x \bar{\otimes} \mathcal{R}, (\alpha \otimes \delta)|_{H_x})$ as per Remark~\ref{rem:Polish-spaces}. 
In this way, we get a measurable field of Polish spaces $X \ni x \mapsto P_x$. Note that by assumption the sets $P_x$ are all non-empty and hence, there exists some measurable section 
$ X \ni x \mapsto (\theta_x, \mathbbm{v}_x) \in P_x$.
Defining $w_{(g,x)} := \mathbbm{v}_{x}(g)$ for $g \in H_x$, we get that --- although $w$ is not defined on all of $\mathcal{G}$ yet --- the equations \eqref{eq:required_cc}--\eqref{eq:required_cocycle} are satisfied whenever they make sense.
In the rest of the proof we will show with the help of Lemma \ref{lemma:cohomology_lemma} that there exists a (potentially different) section for which there exists a well-defined map $w$ on all of $\mathcal{G}$ obeying conditions \eqref{eq:required_cc}--\eqref{eq:required_cocycle}.

Denote by $\mathcal{S}$ the countable non-singular orbit equivalence relation associated to $\mathcal{G}$, i.e.,
\[\mathcal{S}=\{(g \cdot x,x) \mid (g,x) \in \mathcal{G}\}.\]
As $G \acts (X,\mu)$ is amenable, the relation $\mathcal{S}$ is amenable and hence it follows by the Connes--Feldman--Weiss theorem \cite{ConnesFeldmanWeiss81} that after neglecting a set of measure zero, there exists a partition of $X$ into $\mathcal{S}$-invariant Borel subsets $X_0\sqcup X_1$ such that the restriction of $\mathcal{S}$ to $X_0$ has finite orbits and the restriction to $X_1$ is induced by a free non-singular action of $\Z$. 
This implies that the map $q \colon \mathcal{G} \rightarrow \mathcal{S}\colon (g,x) \mapsto (g \cdot x,x)$ admits a measurable section, i.e., a measurable groupoid morphism $s \colon \mathcal{S} \rightarrow \mathcal{G}$ such that $q \circ s = \mathrm{id}_\mathcal{S}$. 
Therefore, we can view $\mathcal{G}$ as the semi-direct product of the field of groups $(H_x)_{x \in X}$ and the measurable field of group isomorphisms $\phi_{(x,y)}\colon H_y \rightarrow H_x$ given by $\phi_{(x,y)}(g)= s(x,y) g s(x,y)^{-1}$.
Note that $\phi_{(x,y)} \circ \phi_{(y,z)} = \phi_{(x,z)}$ for all $(x,y), (y,z) \in \mathcal{S}$.

This means that in order to define a measurable field $\mathcal{G} \ni (g,x) \mapsto w_{(g,x)} \in \mathcal{U}(B_{g \cdot x} \bar{\otimes} \mathcal{R})$ satisfying \eqref{eq:required_cocycle}, it suffices to find measurable families of unitaries 
${v{(x,y)} \in \mathcal{U}(B_x \bar{\otimes} \mathcal{R})}$ for $(x,y) \in \mathcal{S}$ and 
$\mathbbm{v}_x(g) \in \mathcal{U}(B_x \bar{\otimes} \mathcal{R})$ for $x \in X, g \in H_x$ such that
\begin{itemize}
\item $v$ is a cocyle for the action of $\mathcal{S}$ on the field of factors $(B_x\bar{\otimes} \mathcal{R})_{x \in X}$ induced by $s$, i.e.,
$v{(x,y)} (\alpha \otimes \delta)_{s(x,y)}(v{(y,z)})=v(x,z)$ for all $(x,y), (y,z) \in \mathcal{S}$;
\item for each $x \in X$, the family $(\mathbbm{v}_x(g))_{g \in H_x}$ defines a cocycle for the action $H_x \acts B_x \bar{\otimes} \mathcal{R}$;
\item $\mathbbm{v}_x(g) = v{(x,y)} (\alpha \otimes \delta)_{s(x,y)}\left(\mathbbm{v}_y\left({\phi_{(y,x)}(g)}\right)\right) (\alpha \otimes \delta)_g\left(v(x,y)^*\right) $ for all $(x,y) \in \mathcal{S}$ and $g \in H_x$. 
\end{itemize}
If these conditions are met, then setting $w_{g s(x,y)} := \mathbbm{v}_x(g)(\alpha \otimes \delta)_g (v(x,y))$ for $(x,y) \in \mathcal{S}$ and $g \in H_x$ yields condition \eqref{eq:required_cocycle}.
Moreover, in order for a measurable field of $*$-isomorphisms ${X \ni x \mapsto \theta_x \colon B_x \rightarrow B_x \bar{\otimes} \mathcal{R}}$ to satisfy \eqref{eq:required_cc}, it then suffices to check that 
\begin{itemize}
\item for each $x \in X$, the pair $(\theta_x, \mathbbm{v}_x)$ is a cocycle conjugacy in $P_x$;
\item for $(x,y) \in \mathcal{S}$ we have $\theta_x \circ \alpha_{s(x,y)} \circ \theta_y^{-1} = \mathrm{Ad}(v(x,y))\circ (\alpha \otimes \delta)_{s(y,z)}$.
\end{itemize}
We introduce some notation to rephrase this in terms of the terminology of Lemma~\ref{lemma:cohomology_lemma}. Consider the natural action $\mathcal{U}(B_x \bar{\otimes} \mathcal{R}) \acts P_x$ given by composing a cocycle conjugacy with the inner one given by $\mathrm{Ad}(u)$ for $u \in \mathcal{U}(B_x \bar{\otimes} \mathcal{R})$ as per Remark~\ref{rem:Polish-spaces}.
In this way, we get a measurable field $(\mathcal{U}(B_x \bar{\otimes} \mathcal{R}) \acts P_x)_{x \in X}$ of continuous actions of Polish groups on Polish spaces.
Let us convince ourselves that in the terminology of Lemma~\ref{lemma:cohomology_lemma}, the equivalence relation $\mathcal{S}$ acts by conjugacies on this field of actions.
Firstly, we have a measurable field of group isomorphisms 
\[
{\mathcal{S} \ni (x,y) \mapsto  \gamma_{(x,y)} = (\alpha \otimes \delta)_{s(x,y)}|_{\mathcal{U}(B_y \bar{\otimes} \mathcal{R})} \colon \mathcal{U}(B_y \bar{\otimes} \mathcal{R})\rightarrow \mathcal{U}(B_x \bar{\otimes} \mathcal{R})}
\]
such that $\gamma_{(x,y)} \circ \gamma_{(y,z)} = \gamma_{(x,z)}$ for all $(x,y), (y,z) \in \mathcal{S}$.
The latter formula holds as $\alpha\otimes\delta$ was a $\mathcal G$-action and $s: \mathcal{S}\to\mathcal{G}$ is a section.
Secondly, we have an action of $\mathcal{S}$ on $(P_x)_{x \in X}$ as follows.
Given $(x,y) \in \mathcal{S}$ and $(\theta, \mathbbm{v}) \in P_y$, we define $\beta_{(x,y)}(\theta, \mathbbm{v}) := (\tilde{\theta}, \tilde{\mathbbm{v}})$, where
\[
\tilde{\theta} = (\alpha \otimes \delta)_{s(x,y)} \circ \theta \circ \alpha_{s(y,x)} \quad\text{and}\quad
\tilde{\mathbbm{v}}(h) = (\alpha \otimes \delta)_{s(x,y)}(\mathbbm{v}({\phi_{(y,x)}(h)})) \text{ for } h \in H_x.
\]
This construction yields a well-defined cocycle conjugacy in $P_x$, and we get a well-defined map $\beta_{(x,y)}\colon P_y \rightarrow P_x$.
Together these maps combine to form a measurable field of homeomorphisms $\mathcal{S} \ni (x,y) \mapsto \beta_{(x,y)}\colon P_y \to P_x$ such that $\beta_{(x,y)} \circ \beta_{(y,z)} = \beta_{(x,z)}$ for all $(x,y), (y,z) \in \mathcal{S}$. 
This formula holds once again because $\alpha\otimes\delta$ and $\alpha$ are $\mathcal G$-actions and $s\colon \mathcal{S}\to\mathcal{G}$ is a section.
Moreover, the maps $\beta$ and $\gamma$ are compatible with the measurable field of actions ${(\mathcal{U}(B_x \bar{\otimes} \mathcal{R}) \acts P_x)_{x \in X}}$ (as required for Lemma~\ref{lemma:cohomology_lemma}), since for any $(x,y) \in \mathcal{S}, u \in \mathcal{U}(B_y \bar{\otimes} \mathcal{R})$ and $(\theta,\mathbbm{v}) \in P_y$ we may simply compare definitions and observe
\[
\beta_{(x,y)}(u \cdot (\theta,\mathbbm{v})) = \gamma_{(x,y)}(u) \cdot \beta_{(x,y)}(\theta,\mathbbm{v}).
\]
Having introduced all this data, our previous discussion can be rephrased. 
In order to complete the proof, it suffices to find a measurable section $X \ni x \mapsto \sigma(x) \in P_x$ and a measurable family $\mathcal{S} \ni (x,y) \mapsto v(x,y) \in \mathcal{U}(B_x \bar{\otimes} \mathcal{R})$ such that
\begin{itemize}
\item $v(x,y) \gamma_{(x,y)}(v(y,z)) = v(x,z)$ for all $(x,y), (y,z) \in \mathcal{S}$;
\item $v(x,y) \cdot \beta_{(x,y)}(\sigma(y)) = \sigma(x)$ for all $(x,y) \in \mathcal{S}$.
\end{itemize}
By Lemma~\ref{lemma:cohomology_lemma}, such maps exist if we can merely show that
there exists a measurable section $X \ni x \mapsto (\theta_x, \mathbbm{v}_x) \in P_x$ such that for all $(x,y) \in \mathcal{S}$, the element $(\theta_x, \mathbbm{v}_x)$ belongs to the closure of $\mathcal{U}(B_x \bar{\otimes}\mathcal{R}) \cdot \beta_{(x,y)}(\theta_y,\mathbbm{v}_y)$.
We claim that this is indeed the case.

Consider any measurable section $X \ni x \mapsto (\theta'_x, \mathbbm{v}'_x) \in P_x$. 
As $\delta$ is strongly self-absorbing, we can fix a cocycle conjugacy $(\Phi, \mathbb{U})$ from $(\mathcal{R}, \delta)$ to $(\mathcal{R}\bar{\otimes} \mathcal{R}, \delta \otimes \delta)$ that is approximately unitarily equivalent to $\operatorname{id}_{\mathcal R}\otimes 1_{\mathcal R}$.
For each $x \in X$ we can define the map
\[
\Lambda_x\colon P_x \rightarrow P_x,\quad (\theta, \mathbbm{v}) \mapsto \big( (\theta, \mathbbm{v})^{-1} \otimes \mathrm{id}_\mathcal{R} \big) \circ \big( \mathrm{id}_{B_x} \otimes (\Phi, \mathbb{U}) \big) \circ (\theta, \mathbbm{v})
\]
Then we get a new measurable section
\[
X \ni x \mapsto (\theta_x, \mathbbm{v}_x) := \Lambda_x(\theta'_x, \mathbbm{v}'_x) \in P_x.
\]
We claim that this section does the trick. Fix $(x,y) \in \mathcal{S}$. 
If we denote $(\tilde{\theta}_x, \tilde{\mathbbm{v}}_x):= \beta_{(x,y)}(\theta_y, \mathbbm{v}_y)$, we need to convince ourselves that the cocycle conjugacy $(\theta_x, \mathbbm{v}_x)$ is in the closure of $\mathcal{U}(B_x \bar{\otimes}\mathcal{R})\cdot (\tilde{\theta}_x, \tilde{\mathbbm{v}}_x)$.
First of all, we observe that the construction of $(\theta_x, \mathbbm{v}_x)=\Lambda_x(\theta_x', \mathbbm{v}_x')$ can be seen as a special case of Lemma~\ref{lem:special-cc} for $M_1=M_2=B_x$, $G=H_x$, $\phi=\operatorname{id}_{H_x}$ and $\Delta=\operatorname{id}_{\mathcal R}$.
Hence we can find a sequence of unitaries $y_n\in\mathcal{U}(B_x\otimes\mathcal R)$ such that
\begin{equation} \label{eq:main-tech:yn}
y_n^*\theta_x(b)y_n \to b\otimes 1_{\mathcal R},\quad y_n^*\mathbbm{v}_x(h)(\alpha\otimes\delta)_h(y_n) \to 1_{B_x\otimes\mathcal R}
\end{equation}
in the strong-$*$ operator topology for all $b\in B_x$ and $h\in H_x$.

Next, by our previous notation, the group isomorphism $\phi_{(x,y)}\colon H_y\to H_x$ is exactly the one so that the isomorphism of von Neumann algebras $\alpha_{s(x,y)}\colon B_y\to B_x$ can be viewed as a (genuine) conjugacy between the $H_x$-actions $(\alpha|_{H_y})_{\phi_{(y,x)}}$ and $\alpha|_{H_x}$.
Moreover if $s(x,y)=(k,y)$ for $k\in G$, then $(\alpha\otimes\delta)_{s(x,y)}=\alpha_{s(x,y)}\otimes\delta_k$ and $\delta_k$ can be seen as a conjugacy between the $H_x$-actions $(\delta|_{H_y})_{\phi_{(y,x)}}$ and $\delta|_{H_x}$.

By definition of $\beta_{(x,y)}$, we have
\[
\begin{array}{ccl}
\tilde{\theta}_x &=& (\alpha \otimes \delta)_{s(x,y)} \circ \theta_y \circ \alpha_{s(y,x)} \\
&=&  (\alpha \otimes \delta)_{s(x,y)} \circ \big( {\theta_y'}^{-1} \otimes \mathrm{id}_\mathcal{R} \big) \circ \big( \mathrm{id}_{B_y} \otimes \Phi \big) \circ \theta_y' \circ \alpha_{s(y,x)} \\
&=& \big( ({\theta_y'}\circ\alpha_{s(y,x)})^{-1} \otimes \delta_k \big) \circ \big( \mathrm{id}_{B_y} \otimes \Phi \big) \circ (\theta_y' \circ \alpha_{s(y,x)})
\end{array}
\]
and for all $h\in H_x$ with $g=\phi_{(y,x)}(h)$, one has
\[
\begin{array}{ccl}
\tilde{\mathbbm{v}}_x(h) 
&=& (\alpha \otimes \delta)_{s(x,y)}\Big( ({\theta_y'}^{-1}\otimes\operatorname{id}_{\mathcal R})\Big( (\operatorname{id}_{B_y}\otimes\Phi)(\mathbbm{v}_y'(g)) \cdot (1_{B_y}\otimes \mathbb{U}_g)\cdot (\mathbbm{v}_y'(g)^* \otimes1_{\mathcal R}) \Big) \Big) \\
&=& \big( ({\theta_y'}\circ\alpha_{s(y,x)})^{-1}\otimes\delta_k \big)\Big( (\operatorname{id}_{B_y}\otimes\Phi)(\mathbbm{v}_y'(g)) \cdot (1_{B_y}\otimes \mathbb{U}_g)\cdot (\mathbbm{v}_y'(g)^* \otimes1_{\mathcal R}) \Big)
\end{array}
\]
We conclude that Lemma~\ref{lem:special-cc} is applicable to the cocycle conjugacy $(\tilde{\theta}_x,\tilde{\mathbbm{v}}_x)$, where we insert $G_1=H_x$, $G_2=H_y$, $\phi=\phi_{(y,x)}$, $M_1=B_x$, $M_2=B_y$, $\Delta=\delta_k$ and the cocycle conjugacy $(\theta'_y\circ\alpha_{s(y,x)}, \mathbbm{v}'_y\circ\phi_{(y,x)})$ in place of $(\theta,\mathbbm{v})$.
This allows us to find a sequence of unitaries $w_n\in\mathcal{U}(B_x\otimes\mathcal R)$ satisfying
\begin{equation} \label{eq:main-tech:wn}
w_n(b\otimes 1_{\mathcal R})w_n^* \to \tilde{\theta}_x(b),\quad w_n(\alpha\otimes\delta)_h(w_n)^* \to \tilde{\mathbbm{v}}_x(h)
\end{equation}
in the strong-$*$ operator topology for all $b\in B_x$ and $h\in H_x$.
If we consider both of the conditions \eqref{eq:main-tech:yn} and \eqref{eq:main-tech:wn} and keep in mind that $G$ is countable and $B_x$ is separable and semi-finite, we can apply Lemma~\ref{lemma:point_strong_dense_subset} and find an increasing sequence of natural numbers $m_n$ such that the resulting sequence of unitaries $z_n=w_ny_{m_n}^*$ satisfies
\[
z_n\theta_x(b)z_n^* \to \tilde{\theta}_x(b),\quad z_n\mathbbm{v}_x(h)(\alpha\otimes\delta)_h(z_n)^* \to \tilde{\mathbbm{v}}_x(h)
\]
in the strong-$*$ operator topology for all $b\in B_x$ and $h\in H_x$.
Then it follows from Remark~\ref{rem:Polish-spaces} that $(\theta_x, \mathbbm{v}_x)$ indeed belongs to the closure of $\mathcal{U}(B_x \bar{\otimes}\mathcal{R})\cdot (\tilde{\theta}_x, \tilde{\mathbbm{v}}_x)$.
This finishes the proof.
\end{proof}

\begin{definition}[see {\cite[Definition~3.4]{Delaroche79}}]
An action $\alpha \colon G \acts B$ of a countable discrete group on a von Neumann algebra is called amenable, if there exists an equivariant conditional expectation 
\[
P\colon (\ell^\infty(G) \bar{\otimes} B, \tau \otimes \alpha) \rightarrow (B, \alpha),
\]
where $\tau$ denotes the left translation action $G \acts \ell^\infty(G)$.
\end{definition}

By \cite{Delaroche82}, an action $\alpha$ as above is amenable if and only if its restriction to $\mathcal{Z}(B)$ is amenable, which is equivalent to the action on the measure-theoretic spectrum of the center being amenable in the sense of Zimmer.
Recall that an automorphism $\alpha\in\operatorname{Aut}(M)$ on a separably acting von Neumann algebra is properly centrally non-trivial, if for every non-zero projection $p\in\mathcal{Z}(M)$, the restriction of $\alpha_\omega$ on $pM_\omega$ is non-trivial.

The following result contains Theorem~\ref{theorem-A} for actions on semi-finite von Neumann algebras as a special case via $H=G$.

\begin{corollary} \label{cor:McDuff-passes}
Let $G$ be a countable discrete group and $B$ a semi-finite von Neumann algebra with separable predual such that $B \cong B \bar{\otimes} \mathcal{R}$. 
Let $\alpha: G \acts B$ be an amenable action.
Suppose $H\subseteq G$ is a normal subgroup such that for every $g\in G\setminus H$, the automorphism $\alpha_g$ is properly centrally non-trivial.
Let $G_1=G/H$ with quotient map $\pi: G\to G_1$ and let $\delta: G_1\curvearrowright\mathcal R$ be a strongly self-absorbing action.
Then $\alpha \simeq_{\mathrm{cc}} \alpha \otimes \delta_\pi$.
\end{corollary}
\begin{proof}
Adopt the notation from Example~\ref{ex:central_decomposition} and Theorem~\ref{thm:main_technical}.
We identify $\alpha$ with an action of ${\mathcal{G}:= G \ltimes X}$ on a measurable field of semi-finite factors with separable predual $(B_x)_{x \in X}$. 
Denote by ${X \ni x \mapsto H_x}$ the measurable field of isotropy groups.
Amenability of the action $\alpha$ implies that the action on $(X,\mu)$ is amenable in the sense of Zimmer, 
which in turn implies amenability of the associated transformation groupoid. 
In particular almost all isotropy groups $H_x$ are amenable. 

By assumption on $H$ and \cite[Theorem 9.14]{MasudaTomatsu16}, it follows for every $g\in G\setminus H$ and $\mu$-almost all $x\in X$ that either $g\notin H_x$ or the automorphism $(\alpha|_{H_x})_g$ on the McDuff factor $B_x$ is centrally non-trivial.
In other words, after discarding a null set from $X$, we may assume for all $x\in X$ that for all $h\in H_x\setminus(H_x\cap H)$, the automorphism $(\alpha|_{H_x})_g$ on $B_x$ is centrally non-trivial.
By Theorem~\ref{theorem:model-absorption}, we get that $(\alpha|_{H_x})$ is cocycle conjugate to $(\alpha\otimes\delta_\pi)|_{H_x}$.
The claim then follows via Theorem~\ref{thm:main_technical}.
\end{proof}


\section{Actions on arbitrary von Neumann algebras}

In this section we shall generalize some of the main results we obtained so far, namely Corollaries~\ref{cor:equivalence_equivariant_McDuff} and \ref{cor:McDuff-passes}, to the context of group actions on not necessarily semi-finite von Neumann algebras.
This uses standard results in Tomita--Takesaki theory, which allow us to reduce the generalize case to the semi-finite case considered in the previous sections.
We will henceforth assume that the reader is familiar with the basics of Tomita--Takesaki theory as well as the theory of crossed products (for a thorough treatment the reader should consult the book \cite{Takesaki03}), although we are about to recall the specific points needed about the former for this section.

\begin{remark}[see {\cite[Chapters VIII, XII]{Takesaki03}}] \label{rem:tt-basics}
Let $M$ be a separably acting von Neumann algebra.
Given a faithful normal state $\varphi$ on $M$, we denote by $\sigma^\varphi: \mathbb{R}\curvearrowright M$ its associated \emph{modular flow}.
If $\psi$ is another faithful normal state on $M$, we denote by $(D\psi: D\varphi): \mathbb{R}\to\mathcal{U}(M)$ the associated \emph{Connes cocycle}, which is a $\sigma^\varphi$-cocycle satisfying $\operatorname{Ad}(D\psi: D\varphi)_t\circ\sigma^\varphi_t=\sigma^\psi_t$ for all $t\in\mathbb{R}$.
The crossed product von Neumann algebra $\tilde{M}=M\rtimes_{\sigma^\varphi}\mathbb{R}$ is called the \emph{continuous core} of $M$ and does not depend on the choice of $\varphi$ up to canonical isomorphism.
With slight abuse of notation, we will consider $M$ as a von Neumann subalgebra in $\tilde{M}$, and denote by $\lambda^{\sigma^\varphi}: \mathbb{R}\to\mathcal{U}(\tilde{M})$ the unitary representation implementing the modular flow on $M$.
The continuous core $\tilde{M}$ is always a semi-finite von Neumann algebra.
Given any automorphism $\alpha\in\operatorname{Aut}(M)$, there is an induced \emph{extended automorphism} $\tilde{\alpha}\in\operatorname{Aut}(\tilde{M})$ uniquely determined by
\[
\tilde{\alpha}|_M = \alpha \quad\text{and}\quad \tilde{\alpha}(\lambda^{\sigma^\varphi}_t)=(D \varphi\circ\alpha^{-1} : D\varphi)_t\cdot\lambda^{\sigma^\varphi}_t,\quad t\in\mathbb R.
\]
The assignment $\alpha\mapsto\tilde{\alpha}$ defines a continuous homomorphism of Polish groups.
Therefore, given more generally a continuous action $\alpha: G\curvearrowright M$ of a second-countable locally compact group, we may induce the \emph{extended action} $\tilde{\alpha}: G\curvearrowright\tilde{M}$.
Every extended automorphism on $\tilde{M}$ has the property that it commutes with the dual flow $\hat{\sigma}^\varphi: \mathbb{R}\curvearrowright\tilde{M}$.
\end{remark}

In the proof of Theorem \ref{thm:general-McDuff} we will appeal to the following proposition, but only in a special case.
In that particular instance, we note that the needed result can also be deduced from \cite[Lemma 3.3]{Tomatsu17}.

\begin{prop} \label{prop:cp-centralizer}
Let $G$ be a second-countable locally compact group.
Let $\alpha: G\curvearrowright M$ be an action on a separably acting von Neumann algebra.
Then the normal inclusion $M\subset M\rtimes_\alpha G$ induces (via componentwise application of representing sequences) a unital $*$-homomorphism
$(M_{\omega,\alpha})^{\alpha_\omega} \to (M\rtimes_\alpha G)_\omega$.
\end{prop}
\begin{proof}
Assume $M$ is represented faithfully on a Hilbert space $\mathcal{H}$, and consider the canonical inclusion $\pi\colon M \rtimes_\alpha G \rightarrow \mathcal{B}(L^2(G)) \bar{\otimes} M \subset \mathcal{B}(L^2(G, \mathcal{H}))$, which on $x \in M \subset M \rtimes_\alpha G$ is given by
\[ [\pi(x) \xi](g) = \alpha_g^{-1}(x)\xi(g), \quad \xi \in L^2(G, \mathcal{H}), g \in G.\]
If $(x_n)_{n \in \N}$ is any bounded sequence in $M$ representing an element $x \in (M_{\omega, \alpha})^{\alpha_\omega}$, then the invariance of $x$ will guarantee that
$\pi(x_n) - (1 \otimes x_n) \rightarrow 0$ in the strong-$*$ operator topology as $n \rightarrow \omega$. Since $(1 \otimes x_n)_{n \in \N}$ represents an element in $(\mathcal{B}(L^2(G)) \bar{\otimes} M)_\omega$, it follows that $(\pi(x_n))_{n \in \N}$ represents an element in $(\pi(M\rtimes_\alpha G))_\omega$, so $x \in (M\rtimes_\alpha G)_\omega$. This finishes the proof.
\end{proof}

\begin{prop} \label{prop:Takai-duality}
Let $G$ be a second-countable locally compact group and $\alpha: G \acts M$ an action on a von Neumann algebra with separable predual.
Let $\varphi$ be a faithful normal state on $M$.
Let $\tilde{\alpha}: G\curvearrowright\tilde{M}=M\rtimes_{\sigma^\varphi}\mathbb R$ be the extended action on the continuous core as in Remark~\ref{rem:tt-basics}.
With some abuse of notation, denote by $\tilde{\alpha}$ also the induced action $G\curvearrowright\tilde{M}\rtimes_{\hat{\sigma}^\varphi}\mathbb R$ by acting trivially on the canonical unitary representation implementing $\hat{\sigma}^\varphi$.
Then under the Takesaki--Takai duality isomorphism $\tilde{M}\rtimes_{\hat{\sigma}^\varphi}\mathbb R \cong \mathcal B(L^2(\mathbb R))\bar{\otimes} M$, the action $\tilde{\alpha}$ is cocycle conjugate to $\operatorname{id}_{\mathcal{B}(L^2(\mathbb R))}\otimes\alpha$.
\end{prop}
\begin{proof}
Denote $\alpha'=\operatorname{id}_{\mathcal{B}(L^2(\mathbb R))}\otimes\alpha$.
We apply Takesaki--Takai duality \cite[Chapter X, Theorem 2.3(iii)]{Takesaki03} to understand the $G$-action $\tilde{\alpha}$.
If $M$ is represented faithfully on a Hilbert space $\mathcal H$, then the natural isomorphism 
\[
\Theta: \tilde{M}\rtimes_{\hat{\sigma}^\varphi}\mathbb R = (M\rtimes_{\sigma^\varphi}\mathbb R)\rtimes_{\hat{\sigma}^\varphi}\mathbb R \to \mathcal{B}(L^2(\mathbb R)) \bar{\otimes} M\subseteq \mathcal{B}(L^2(\mathbb R,\mathcal H))
\]
has the following properties.
Let $\xi\in L^2(\mathbb R,\mathcal H)$.
For $x\in M$ and $s,t\in\mathbb R$ we have
\[
[\Theta(x)\xi ](s) = \sigma^\varphi_{-s}(x)\xi(s) \quad\text{and}\quad [\Theta(\lambda^{\sigma^\varphi}_t)\xi](s) = \xi(s-t).
\]
Furthermore, if the dual flow is given via the convention 
$\hat{\sigma}^\varphi_t(\lambda_s^{\sigma^{\varphi}}) = e^{its}\lambda^{\sigma^{\varphi}}_s$, then we also have
\[
[\Theta(\lambda^{\hat{\sigma}^\varphi}_t)\xi](s)=e^{its}\xi(s).
\]
We consider a continuous family of unitaries $G\ni g\mapsto \mathbb{W}_g\in \mathcal{B}(L^2(\mathbb R)) \bar{\otimes} M$ given by
\[
[\mathbb{W}_g\xi](s) = (D\varphi: D\varphi\circ\alpha_g^{-1})_{-s}\xi(s),\quad \xi\in L^2(\mathbb R,\mathcal H),\quad s\in\mathbb R.
\]
We claim that this defines an $\alpha'$-cocycle.
Indeed, by using the chain rule for the Connes cocycle (\cite[Theorem VIII.3.7]{Takesaki03}), we observe for all $g,h\in G$, $\xi\in L^2(\mathbb R,\mathcal H)$ and $s\in\mathbb R$ that
\[
\begin{array}{ccl}
[\mathbb{W}_g\alpha'_g(\mathbb W_h)\xi](s) &=& (D\varphi: D\varphi\circ\alpha_g^{-1})_{-s}[\alpha'_g(\mathbb W_h)\xi](s) \\
&=& (D\varphi: D\varphi\circ\alpha_g^{-1})_{-s}\alpha_g\big( (D\varphi: D\varphi\circ\alpha_h^{-1})_{-s} \big)\xi(s) \\
&=& (D\varphi: D\varphi\circ\alpha_g^{-1})_{-s}\cdot (D\varphi\circ\alpha_g^{-1}: D\varphi\circ\alpha_{gh}^{-1})_{-s} \big)\xi(s) \\
&=& (D\varphi: D\varphi\circ\alpha_{gh}^{-1})_{-s}\xi(s) \ = \ [\mathbb{W}_{gh}\xi](s)
\end{array}
\]
Given how $\tilde{\alpha}$ acts on the domain of $\Theta$, we can observe (using \cite[Corollary VIII.1.4]{Takesaki03}) for any $g\in G$, $x \in M$, $\xi \in L^2(\R, \mathcal{H})$ and $s \in \R$ that
\[
\begin{array}{ccl}
[\Theta(\tilde{\alpha}_g(x))\xi ](s) &=& \sigma^\varphi_{-s}(\alpha_g(x))\xi(s) \\
&=& \operatorname{Ad}(D\varphi: D\varphi\circ\alpha_g^{-1})_{-s}\Big( \sigma_{-s}^{\varphi\circ\alpha_g^{-1}}(\alpha_g(x)) \Big)\xi(s) \\
&=& \Big(\operatorname{Ad}(D\varphi: D\varphi\circ\alpha_g^{-1})_{-s}\circ\alpha_g\Big) \big( \sigma^\varphi_{-s}(x) \big)\xi(s) \\
&=& \Big[ \Big( \operatorname{Ad}(\mathbb{W}_g)\circ\alpha'_g\circ\Theta\Big)(x)\, \xi \Big](s).
\end{array}
\]
Moreover, using the cocycle identity and the chain rule again, we can see for any $g \in G$, $\xi \in L^2(\R,\mathcal{H})$ and $s,t \in \R$ that
\[
\begin{array}{ccl}
[\Theta(\tilde{\alpha}_g(\lambda^{\sigma^\varphi}_t))\xi](s) &=& [\Theta\big( (D \varphi\circ\alpha_g^{-1} : D\varphi)_t\cdot\lambda^{\sigma^\varphi}_t \big)\xi](s) \\
&=& \sigma^{\varphi}_{-s}(D \varphi\circ\alpha_g^{-1} : D\varphi)_t [ \Theta(\lambda^{\sigma^\varphi}_t)\xi ](s) \\
&=& \sigma^{\varphi}_{-s}(D \varphi\circ\alpha_g^{-1} : D\varphi)_t \xi(s-t) \\
&=& (D\varphi: D\varphi\circ\alpha_g^{-1})_{-s} (D\varphi: D\varphi\circ\alpha_g^{-1})_{t-s}^* \xi(s-t) \\
&=& (D\varphi: D\varphi\circ\alpha_g^{-1})_{-s} \Big[ \Theta(\lambda_t^{\sigma^\varphi}) \mathbb{W}_g^* \xi \Big](s) \\
&=& \Big[ \mathbb{W}_g \Theta(\lambda_t^{\sigma^\varphi}) \mathbb{W}_g^* \xi \Big](s) \\
&=& \Big[ \Big( \operatorname{Ad}(\mathbb{W}_g)\circ\alpha'_g\circ\Theta\Big)(\lambda_t^{\sigma^\varphi}) \xi \Big](s)
\end{array}
\]
Here we used that $\alpha'_g$ fixes the shift operator given by $\Theta(\lambda^{\sigma^\varphi}_t)$ for all $g\in G$ and $t\in\mathbb{R}$.
Lastly, it is trivial to see that $\alpha'$ fixes operators of the form $\Theta(\lambda^{\hat{\sigma}^\varphi}_t)$, which in turn also commute pointwise with the cocycle $\mathbb{W}$.
In conclusion, all of these observations culminate in the fact that the isomorphism $\Theta$ extends to the cocycle conjugacy $(\Theta,\mathbb{W})$ between $\tilde{\alpha}$ and $\alpha'$.
\end{proof}

The following represents our most general McDuff-type absorption theorem:

\begin{theorem} \label{thm:general-McDuff}
Let $G$ be a second-countable locally compact group.
Let $\alpha: G \acts M$ be an action on a von Neumann algebra with separable predual and let $\delta: G \acts \mathcal{R}$ be a strongly self-absorbing action on the hyperfinite II$_1$ factor.
Then $\alpha\simeq_{\mathrm{cc}}\alpha\otimes\delta$ if and only if there exists a unital equivariant $*$-homomorphism $(\mathcal R,\delta) \to (M_{\omega,\alpha},\alpha_\omega)$.
\end{theorem}
\begin{proof}
The ``only if'' part follows exactly as in the proof of Corollary~\ref{cor:equivalence_equivariant_McDuff}, so we need to show the ``if'' part.
Since Corollary~\ref{cor:equivalence_equivariant_McDuff} already covers the case when $M$ is finite, we may split off the largest finite direct summand of $M$ and assume without loss of generality that $M$ is properly infinite.

Let $\varphi$ be a faithful normal state on $M$.
As in Remark~\ref{rem:tt-basics}, we consider the (semi-finite) continuous core $\tilde{M}$ and the extended $G$-action $\tilde{\alpha}: G\curvearrowright\tilde{M}$.
Since the image of $\tilde{\alpha}$ commutes with the dual flow $\hat{\sigma}^{\varphi}$, we have a continuous action $\beta=\tilde{\alpha}\times\hat{\sigma}^\varphi: G\times\mathbb{R}\curvearrowright\tilde{M}$ via $\beta_{(g,t)}=\tilde{\alpha}_g\circ\hat{\sigma}^\varphi_t$ for all $g\in G$ and $t\in\mathbb{R}$.
Let us also consider the action $\delta^{\mathbb R}: G\times\mathbb{R}\curvearrowright\mathcal R$ given by $\delta^{\mathbb R}_{(g,t)}=\delta_g$ for all $g\in G$ and $t\in\mathbb{R}$, which is evidently also strongly self-absorbing.

We apply Proposition~\ref{prop:cp-centralizer} to $\mathbb{R}$ in place of $G$ and to the modular flow in place of $\alpha$.
In this context, we note that by \cite[Theorem 4.1, Proposition 4.11]{AndoHaagerup14}, we have that the $M^\omega$ agrees with the $(\sigma^\varphi,\omega)$-equicontinuous part $M^\omega_{\sigma^\varphi}$.
In particular, the induced ultrapower flow $(\sigma^\varphi)^\omega$ on it is continuous and and its restriction to $M_\omega$ is trivial.
So $M_\omega=(M_{\omega,\sigma^\varphi})^{(\sigma^\varphi)_\omega}$ and Proposition~\ref{prop:cp-centralizer} implies that the inclusion $M\subset\tilde{M}$ induces an embedding $M_\omega\to \tilde{M}_\omega$.
Since by definition, one has $\tilde{\alpha}|_M=\alpha$ as $G$-actions, it is clear that bounded $(\alpha,\omega)$-equicontinuous sequences in $M$ become $(\tilde{\alpha},\omega)$-equicontinuous sequences in $\tilde{M}$.
Keeping in mind that the dual flow $\hat{\sigma}^\varphi$ acts trivially on $M$ by definition, the aforementioned inclusion therefore induces an equivariant unital $*$-homomorphism
\[
(M_{\omega,\alpha},\alpha_\omega) \to ( (\tilde{M}_{\omega,\beta})^{(\hat{\sigma}^\varphi)_\omega},\tilde{\alpha}_\omega ).
\]
If we compose this $*$-homomorphism with a given unital equivariant $*$-homomorphism $(\mathcal R,\delta)\to (M_{\omega,\alpha},\alpha_\omega)$, we can view the resulting map as a $(G\times\mathbb{R})$-equivariant unital $*$-homomorphism $(\mathcal R,\delta^{\mathbb R}) \to (\tilde{M}_{\omega,\beta},\beta_\omega)$.
Since $\tilde{M}$ is semi-finite, it follows from Corollary~\ref{cor:equivalence_equivariant_McDuff} that $\beta$ and $\beta\otimes\delta^{\mathbb R}$ are cocycle conjugate as $(G\times\mathbb R)$-actions, say via $(\Psi,\mathbb{V}): (\tilde{M},\beta)\to (\tilde{M}\bar{\otimes}\mathcal{R},\beta\otimes\delta^{\mathbb R})$.
Remembering $\beta=\tilde{\alpha}\times\hat{\sigma}^\varphi$, we consider the $(\hat{\sigma}^\varphi\otimes \mathrm{id}_\mathcal{R})$-cocycle $\mathbbm{w}_t=\mathbb{V}_{(1_G,t)}$ and the $\tilde{\alpha}\otimes \delta$-cocycle $\mathbbm{v}_g=\mathbb{V}_{(g,0)}$.
The cocycle identity for $\mathbb{V}$ then implies the relation 
\begin{equation} \label{eq:cocycle-interaction}
\mathbbm{w}_t(\hat{\sigma}^\varphi_t\otimes\operatorname{id}_{\mathcal R})(\mathbbm{v}_g)=\mathbbm{v}_g(\tilde{\alpha}\otimes\delta)_g(\mathbbm{w}_t)
\end{equation}
for all $g\in G$ and $t\in\mathbb R$.
The cocycle conjugacy of flows $(\Psi,\mathbbm{w})$ induces an isomorphism
\[
\Lambda:=(\Psi,\mathbbm{w})\rtimes\mathbb R: \tilde{M}\rtimes_{\hat{\sigma}^\varphi}\mathbb R \to (\tilde{M}\bar{\otimes}\mathcal R)\rtimes_{\hat{\sigma}^\varphi\otimes\operatorname{id}_{\mathcal R}}\mathbb R \cong \big(\tilde{M}\rtimes_{\hat{\sigma}^\varphi}\mathbb R\big)\otimes\mathcal R
\]
via
\[
\Lambda|_{\tilde{M}} = \Psi \quad\text{and}\quad \Lambda(\lambda^{\hat{\sigma}^\varphi}_t)=\mathbbm{w}_t(\lambda^{\hat{\sigma}^\varphi}_t\otimes 1_{\mathcal R}),\quad t\in\mathbb R.
\]
With slight abuse of notation (as in Proposition~\ref{prop:Takai-duality}), we also denote by $\tilde{\alpha}$ the obvious induced $G$-action on the crossed product $\tilde{M}\rtimes_{\hat{\sigma}^\varphi}\mathbb R$.
Using that $(\Psi,\mathbb{V})$ was a cocycle conjugacy, we observe for all $g\in G$ and $t\in\mathbb{R}$ that 
\[
\operatorname{Ad}(\mathbbm{v}_g)\circ(\tilde{\alpha}\otimes\delta)_g\circ\Lambda|_{\tilde{M}} = \operatorname{Ad}(\mathbb{V}_{(g,0)})\circ(\beta\otimes\delta^{\mathbb R})_{(g,0)}\circ\Psi=\Psi\circ\beta_{(g,0)}=\Psi\circ\tilde{\alpha}_g
\]
and
\[
\begin{array}{ccl}
\big( \operatorname{Ad}(\mathbbm{v}_g)\circ(\tilde{\alpha}\otimes\delta)_g\circ\Lambda \big)(\lambda^{\hat{\sigma}^\varphi}_t) 
&=& \mathbbm{v}_g \big( (\tilde{\alpha}\otimes\delta)_g(\mathbbm{w}_t) (\lambda^{\hat{\sigma}^\varphi}_t\otimes 1_{\mathcal R}) \big) \mathbbm{v}_g^* \\
&\stackrel{\eqref{eq:cocycle-interaction}}{=}& \mathbbm{w}_t(\hat{\sigma}^\varphi_t\otimes\operatorname{id}_{\mathcal R})(\mathbbm{v}_g) (\lambda^{\hat{\sigma}^\varphi}_t\otimes 1_{\mathcal R}) \mathbbm{v}_g^* \\
&=& \mathbbm{w}_t(\lambda^{\hat{\sigma}^\varphi}_t\otimes 1_{\mathcal R}) \ = \ \Lambda(\lambda^{\hat{\sigma}^\varphi}_t).
\end{array}
\]
In conclusion, the pair $(\Lambda,\mathbbm{v})$ defines a cocycle conjugacy between the $G$-actions $\tilde{\alpha}$ on $\tilde{M}\rtimes_{\hat{\sigma}^\varphi}\mathbb R$ and $\tilde{\alpha}\otimes\delta$ on $\big(\tilde{M}\rtimes_{\hat{\sigma}^\varphi}\mathbb R\big)\otimes\mathcal R$.
By Proposition~\ref{prop:Takai-duality}, the action $\tilde{\alpha}$ is cocycle conjugate to $\operatorname{id}_{\mathcal B(\ell^2(\N))}\otimes\alpha$.
Since we assumed $M$ to be properly infinite, it furthermore follows that $\operatorname{id}_{\mathcal B(\ell^2(\N))}\otimes\alpha$ is cocycle conjugate to $\alpha$.\footnote{Although this appears to be well-known, we could not find a good literature source for this precise claim and in this generality. We note, however, that the recent proof of the analogous \cstar-algebraic statement \cite[Proposition 1.4]{GabeSzabo22kp} carries over in an obvious way to this setting.}
Combining all these cocycle conjugacies yields one between $\alpha$ and $\alpha\otimes\delta$.
This finishes the proof.
\end{proof}

The following consequence is our last main result, which generalizes Corollary~\ref{cor:McDuff-passes} to actions on arbitrary von Neumann algebras.

\begin{theorem} \label{thm:general-amenable-McDuff}
Let $G$ be a countable discrete group and $M$ a von Neumann algebra with separable predual such that $M \cong M\bar{\otimes} \mathcal{R}$. 
Then for every amenable action $\alpha: G \acts M$, one has $\alpha \simeq_{\mathrm{cc}} \alpha \otimes \mathrm{id}_\mathcal{R}$.
\end{theorem}
\begin{proof}
Choose a faithful normal state $\varphi$ on $M$.
Recall that the induced faithful normal state $\varphi^\omega$ on $M^\omega$ restricts to a tracial state on $M_\omega$. We denote by $\|\cdot\|_2=\|\cdot\|_{\varphi^\omega}$ the induced tracial 2-norm on $M_\omega$.
Since we assumed $M$ to be McDuff, it follows that $M_\omega$ is locally McDuff in the following sense:
Given any $\|\cdot\|_2$-separable $*$-subalgebra $B\subset M_\omega$, there exists a unital $*$-homomorphism $\mathcal R\to M_\omega\cap B'$.

Now we choose $N_1=\mathcal{Z}(M)$ as a subalgebra of $M_\omega$.
We may then choose a unital $*$-homomorphism $\psi_1: \mathcal{R}\to M_\omega$, and define $N_2$ to be the $\|\cdot\|_2$-closed $*$-subalgebra generated by $N_1$ and the range of $\alpha_{\omega,g}\circ\psi_1$ for all $g\in G$.
After that, we may choose a unital $*$-homomorphism $\psi_2: \mathcal{R}\to M_\omega\cap N_2'$, and define $N_3$ to be the $\|\cdot\|_2$-closed $*$-subalgebra generated by $N_2$ and the range of $\alpha_{\omega,g}\circ\psi_2$ for all $g\in G$.
Carry on inductively like this and obtain an increasing sequence of $\alpha_{\omega}$-invariant separable von Neumann subalgebras $N_k\subseteq M_\omega$.
The $\|\cdot\|_2$-closure $N$ of the union $\bigcup_{k\geq 1} N_k$ is then a separably acting finite von Neumann subalgebra, which is clearly McDuff and $\alpha_{\omega}$-invariant.
Furthermore we have an equivariant inclusion $\mathcal{Z}(M)\subseteq\mathcal{Z}(N)$, which implies (for instance by \cite[Theorem 2.4]{Zimmer78}) that the action $\alpha_\omega$ is amenable on $N$.

By Corollary~\ref{cor:McDuff-passes} (with $H=G$), it follows that $\alpha_{\omega}|_N$ is cocycle conjugate to $(\alpha_\omega|_N)\otimes\operatorname{id}_{\mathcal R}$.
In particular we may find some unital $*$-homomorphism $\mathcal R\to (N_\omega)^{\alpha_\omega}$.
Applying a standard reindexation trick, we may use this to obtain a unital $*$-homomorphism $\mathcal R\to (M_\omega)^{\alpha_\omega}$, which finishes the proof by Theorem~\ref{thm:general-McDuff}.
\end{proof}


\bibliographystyle{gabor}
\bibliography{master3}
\addcontentsline{toc}{section}{References}

\end{document}